\documentclass[11pt]{amsart}

\usepackage{amsmath,amssymb,amscd,amsthm,amsxtra}
\usepackage[dvips]{graphics,epsfig}

\allowdisplaybreaks[2]

\newtheorem{theorem}{Theorem} [section]
\newtheorem{maintheorem}{Theorem}
\newtheorem{lemma}[theorem]{Lemma}
\newtheorem{proposition}[theorem]{Proposition}
\newtheorem{remark}[theorem]{Remark}

\newtheorem{corollary}[theorem]{Corollary}

\DeclareMathOperator*{\intt}{\int}

\DeclareMathOperator{\MAX}{MAX}

\newcommand{\noi}{\noindent}
\newcommand{\Z}{\mathbb{Z}}
\newcommand{\R}{\mathbb{R}}

\newcommand{\T}{\mathbb{T}}

\newcommand{\al}{\alpha}
\newcommand{\dl}{\delta}

\newcommand{\eps}{\varepsilon}

\newcommand{\g}{\gamma}
\newcommand{\G}{\Gamma}
\newcommand{\ld}{\lambda}

\newcommand{\s}{\sigma}
\newcommand{\ft}{\widehat}
\newcommand{\wt}{\widetilde}
\newcommand{\cj}{\overline}
\newcommand{\dx}{\partial_x}
\newcommand{\dt}{\partial_t}

\newcommand{\I}{\hspace{0.5mm}\text{I}\hspace{0.5mm}}
\newcommand{\II}{\text{I \hspace{-2.8mm} I} }

\newcommand{\jb}[1]
{\langle #1 \rangle}

\begin{document}

\title
[ Periodic Stochastic KdV with Additive Noise]
{\bf Periodic Stochastic Korteweg-de Vries Equation
with the Additive Space-Time White Noise}

\author{Tadahiro Oh}

\address{Tadahiro Oh\\
Department of Mathematics\\
University of Toronto\\
40 St. George St, Rm 6290,
Toronto, ON M5S 2E4, Canada}

\email{oh@math.toronto.edu}



\begin{abstract}
We prove the local well-posedness of the periodic stochastic Korteweg-de Vries equation 
with the additive space-time white noise.
In order to treat low regularity of the white noise in space, we consider the Cauchy problem in the Besov-type space 
$\ft{b}^s_{p, \infty}(\mathbb{T})$ for $ s= -\frac{1}{2}+$, $p = 2+$ such that $sp < -1$.
In establishing the local well-posedness,
we use a variant of the Bourgain space adapted to $\ft{b}^s_{p, \infty}(\mathbb{T})$
and establish a nonlinear estimate on the second iteration on the integral formulation.
The deterministic part of the nonlinear estimate also yields 
the local well-posedness of the deterministic KdV in $M(\T)$, 
the space of finite Borel measures on $\T$.
\end{abstract}

\maketitle

\tableofcontents

\section{Introduction}

In this paper, we prove the local well-posedness of the periodic stochastic KdV equation (SKdV)
with the additive space-time white noise:
\begin{equation} \label{SKDV}
\begin{cases}
du + (\dx^3 u + u\dx u) dt = dW\\
u(x, 0) = u_0(x) 
\end{cases}
\end{equation}

\noindent
where $u$ is a real-valued function, 
$(x, t) \in \T \times \R^+$  with $ \T = [0, 2\pi)$,
and  $W (t) = \frac{\partial B}{\partial x}$ is a cylindrical Wiener process on $L^2(\T)$.
With $e_n(x) = \frac{1}{\sqrt{2\pi}}e^{inx}$, 
we have $W (t) =  \beta_0(t) e_0 + \sum_{n \ne 0 } \frac{1}{\sqrt{2}} \beta_n(t) e_n(x)$
where $\{ \beta_n\}_{n \geq 0}$ is a family of mutually independent complex-valued Brownian motions
(here we take $\beta_0$ to be real-valued)
in a fixed probability space $(\Omega, \mathcal{F}, P)$
associated with a filtration $\{\mathcal{F}_t\}_{t \geq 0}$
and $\beta_{-n}(t) = \cj{\beta_n}(t)$ for $n \geq 1$.
Note that $\text{Var}(\beta_n(1)) = 2$ for $n \geq 1$.

In \cite{DDT1}, de Bouard-Debussche-Tsutsumi considered
\begin{equation} \label{SKDV1}
\begin{cases}
du + (\dx^3 u + u\dx u) dt = \phi dW\\
u(x, 0) = u_0(x), 
\end{cases}
\end{equation}

\noindent
where $\phi$ is a bounded linear operator in $L^2(\mathbb{T})$.
 They showed that \eqref{SKDV1} is locally well-posed
when $\phi$ is a Hilbert-Schmidt operator from $L^2(\T)$ to $H^s(\T)$
for $s > -\frac{1}{2}$.
See \cite{DDT1} and the references therein
for the previous works in the periodic and nonperiodic settings.

In our present work, we consider the case when $\phi$ is the identity operator on $L^2(\T)$. 
i.e. we take the additive noise to be the space-time white noise $\frac{\partial ^2 B}{\dt\dx}$,
where $B(x, t)$ is a two parameter Brownian motion on $\mathbb{T} \times \mathbb{R}^+$.
Note that  
$\phi$ is a Hilbert-Schmidt operator from $L^2(\T)$ to $H^s(\T)$
for $s < -\frac{1}{2}$ but not for $s \geq -\frac{1}{2}.$ 

\medskip

Suppose that $u$ is the solution to \eqref{SKDV},
or equivalently to \eqref{SKDV1} with $\phi = \text{Id}$, the identity operator on $L^2(\mathbb{T})$.
Let $v_1(x, t) = u(x + \alpha_0 t, t) - \alpha_0$, where $\alpha_0 =$ the mean of $u_0$.
Then, $v_1$ satisfies \eqref{SKDV} with the mean 0 initial condition $u_0 - \alpha_0$.
Now, let $ \mathbb{P}_0$ be the projection onto the spatial frequency 0, and 
$\mathbb{P}_{n \ne 0} = \text{Id} - \mathbb{P}_0$.
Note that $\mathbb{P}_0 W (t) = \beta_0(t) e_0(x) = \frac{1}{\sqrt{2\pi}} \beta_0(t)$.
By letting $v_2 = v_1 - \frac{1}{\sqrt{2\pi}} \beta_0(t)$, we see that 
$u$ satisfies \eqref{SKDV} if and only if $v_2$ satisfies
\begin{equation*} 
\begin{cases}
dv_2 + (\dx^3 v_2 + (v_2 +\frac{1}{\sqrt{2\pi}} \beta_0(t))\dx v_2) dt = \mathbb{P}_{n \ne 0} dW\\
v_2(x, 0) = u_0(x) -\al_0
\end{cases}
\end{equation*}

\noindent
almost surely since $\beta_0(0) = 0 $ a.s.
By setting $v_3(x, t) = v_2(x + c_\omega(t) , t)$ 
with $c_\omega(t) = \int_0^t \frac{1}{\sqrt{2\pi}}  \beta_0(t') dt'$, 
it follows that $v_3$ satisfies
\begin{equation*} 
\begin{cases}
d v_3 + (\dx^3 v_3 + v_3\dx v_3) dt =  d\wt{W}\\
v_3 (x, 0) = u_0(x) - \alpha_0,
\end{cases}
\end{equation*}

\noindent
where $\wt{W}(x, t) = \sum_{n \ne 0} \frac{1}{\sqrt{2}} \beta_n(t) e_n (x + c_\omega(t)) 
= \sum_{n \ne 0} \frac{1}{\sqrt{2}}  \beta_n(t)  e^{in c_\omega(t)}  e_n (x)$.
i.e. $v_3$ solves \eqref{SKDV1} where 
\begin{equation} \label{PHI}
\phi = \text{diag} (\phi_n ; n \ne 0)
\ \ \text{with} \
\phi_n (t) = e^{in c_\omega(t)} \, \text{ and }
\ c_\omega(t) = \int_0^t \tfrac{1}{\sqrt{2\pi}}  \beta_0(t') dt'
\end{equation}

\noi
(with respect to the basis $\{e_n\}_{n \in \mathbb{Z}}$.)
Moreover, note that $v_3$ has the spatial mean 0
(as long as it exists)
since $e_0 \notin \text{Range} (\phi)$.
Therefore, in the remaining of the paper, we concentrate on studying the local well-posedness of \eqref{SKDV1}
with $\phi$ given by \eqref{PHI}
and the mean 0 initial condition $u_0$,
(which implies that $u$ has the spatial mean 0 as long as it exists.)

\medskip

Recall that $u$ is called a (local-in-time) mild solution  to \eqref{SKDV1}
if $u$ satisfies 
\begin{equation} \label{duhamel1}
u(t) =  S(t) u_0 -\tfrac{1}{2}  \int_0^t S(t - t') \dx u^2(t') d t' 
+  \int_0^t S(t-t') \phi(t') dW(t') 
\end{equation} 

\noindent
at least for $t \in [0, T]$ for some $T > 0$, 
where  $S(t) = e^{-t\dx^3}$.

Note that the first two terms in \eqref{duhamel1} also appear in the deterministic KdV theory.
Thus, we briefly review recent well-posedness results of the periodic (deterministic) KdV:
\begin{equation} \label{KDV}
\begin{cases}
u_t + u_{xxx} +  u u_x  = 0 \\ 
u \big|_{t = 0} = u_0,
\end{cases} (x, t) \in \mathbb{T} \times \mathbb{R}.
\end{equation}

\noi
In \cite{BO1}, Bourgain  introduced a new weighted space-time Sobolev space $X^{s, b}$
whose norm is given by
\begin{equation} \label{Xsb}
\| u \|_{X^{s, b}(\mathbb{T} \times \mathbb{R})} = \| \jb{n}^s \jb{\tau - n^3}^b 
\ft{u}(n, \tau) \|_{L^2_{n, \tau}(\mathbb{Z} \times \R)},
\end{equation}

\noindent
where $\jb{ \: \cdot \:} = 1 + |  \cdot  | $. 
He proved the local well-posedness of \eqref{KDV} in $L^2(\mathbb{T})$
via the fixed point argument, 
immediately yielding the global well-posedness in $L^2(\mathbb{T})$
thanks to the conservation of the $L^2$ norm.
Kenig-Ponce-Vega \cite{KPV4} improved Bourgain's result 
and established the local well-posedness in $H^{-\frac{1}{2}}(\T)$
by establishing the bilinear estimate 
\begin{equation} \label{KPVbilinear}
\| \dx(uv) \|_{X^{s, -\frac{1}{2}}} \lesssim \| u \|_{X^{s, \frac{1}{2}}} \| v \|_{X^{s, \frac{1}{2}}}, 
\end{equation}

\noi
for $s \geq -\frac{1}{2}$ under the mean 0 assumption on $u$ and $v$.
Colliander-Keel-Staffilani-Takaoka-Tao \cite{CKSTT4} proved 
the corresponding global well-posedness result via the $I$-method. 

There are also results on \eqref{KDV} which exploit its complete integrability.
In \cite{BO3}, Bourgain proved the global well-posedness of \eqref{KDV}
in the class $M(\T)$ of measures $\mu$, assuming that 
its total variation $\|\mu\|$ is sufficiently small.
His proof is based on the trilinear estimate on the second iteration of 
the integral formulation of \eqref{KDV}, 
assuming an a priori uniform bound on the Fourier coefficients of the solution $u$ of the form 
\begin{equation} \label{BOO}
\sup_{n\in \mathbb{Z}} |\ft{u}(n, t)| < C
\end{equation}

\noi
for all $t\in \R$. 
Then, he established \eqref{BOO} using the complete integrability.
More recently, Kappeler-Topalov \cite{KT} proved the global well-posedness of the KdV in $H^{-1}(\T)$
via the inverse spectral method. 

There are also results on the necessary conditions on the  regularity 
with respect to smoothness or uniform continuity 
of the solution map $: u_0 \in H^s (\mathbb{T}) \to u(t) \in H^s(\mathbb{T})$.
Bourgain \cite{BO3} showed that if the solution map is $C^3$, 
then $s \geq -\frac{1}{2}$.  
Christ-Colliander-Tao \cite{CCT}
proved that if the solution map is uniformly continuous, 
then $s \geq -\frac{1}{2}$.
(Also, see Kenig-Ponce-Vega \cite{KPV5}.) 
These results, in particular, imply that
we can not hope to have a local-in-time solution of KdV via the fixed point argument in $H^s$, $s < -\frac{1}{2}$.
Recall that, for each fixed $t$,  
the space-time white noise $\frac{\partial ^2 B}{\dt\dx}$ lies in $\cap_{s < -\frac{1}{2}}H^s \setminus H^{-\frac{1}{2}}$
almost surely. 
Hence, these results for KdV can not be applied to study the local well-posedness of \eqref{SKDV}.

\medskip

 Now, let us discuss the spaces which capture 
 the regularities of the spatial and space-time white noise. 
 Recently, we proved the invariance of the (spatial) white noise for the (deterministic) KdV in \cite{OH3}
(also see \cite{OHRIMS})
by first establishing the local well-posedness in 
an appropriate Banach space containing the support of the (spatial) white noise.
Define the Besov-type space 
 via the norm
\begin{equation} \label{Besov}
\| f\|_{\ft{b}^s_{p, \infty}} 
:= \| \ft{f}\|_{b^s_{p, \infty}} = \sup_j \| \jb{n}^s \ft{f}(n) \|_{L^p_{|n|\sim 2^j}}
= \sup_j \Big( \sum_{|n| \sim 2^j} \jb{n}^{sp} |\ft{f}(n)|^p \Big)^\frac{1}{p}.
\end{equation}

\noi
In \cite{OH3}, using  the theory of abstract Wiener spaces,
we showed that $\ft{b}^s_{p, \infty}$ contains the full support of the (spatial) white noise for $sp < -1$.
(The statement also holds true for $sp = -1$.)

Let's consider the stochastic convolution $\Phi(t)$ given by
\begin{equation} \label{stoconv}
\Phi(t) = \int_0^t S(t-t') \phi (t')  dW(t'),
\end{equation}

\noi
where $\phi$ is given by \eqref{PHI}.
Define a variant of the $X^{s, b}$ space adjusted to $\ft{b}^s_{p, \infty} (\mathbb{T})$.
Let $X^{s, b}_{p, q}$ be the completion of the Schwartz class $\mathcal{S}(\mathbb{T} \times \mathbb{R})$ under the norm
\begin{equation} \label{XSBP}
 \| u \|_{X^{s, b}_{p, q}} 
 =  \|\jb{n}^s \jb{\tau - n^3}^b \ft{u}(n, \tau)\|_{b^0_{p, \infty} L^q_\tau}.
\end{equation}

\noi
Note that $X^{s, b}_{p, q}$ defined in \eqref{XSBP} is the space of functions $u$ such that 
$S(-t)u(\cdot, t) \in (\ft{b}^{s}_{p, \infty})_x  (\mathcal{F}L^{b, q})_t$,
where $\mathcal{F}L^{b, q}$ is defined via the norm 
\begin{equation}\label{FLBP}
 \|f\|_{\mathcal{F}L^{b, q}} :=  \| \jb{\tau}^b \ft{f}(\tau)\|_{L^q}.
 \end{equation}

\noi
In \cite{OH3}, we also showed that the local-in-time white noise is supported on 
$\mathcal{F}L^{c, q}$ for $ c q < -1 $.
This implies that the Brownian motion belongs locally in time to 
$\mathcal{F}L^{b, q}$ for $ (b-1) q < -1$.
Hence, with $b <  \frac{1}{2}$ and $q = 2$, 
we see that the local-in-time 
stochastic convolution $\eta(t) \Phi(t)$ lies in $X^{s, b}_{p, q}$ almost surely,
with $sp < -1$, $b  <  \frac{1}{2}$ and $q = 2$, 
where $\eta(t)$ is a smooth cutoff supported on $[-1, 2]$ with $\eta(t) \equiv 1$ on $[0, 1]$.

\medskip

The argument by de Bouard-Debussche-Tsutsumi \cite{DDT1} is based on the result by Roynette \cite{ROY}
on the endpoint regularity of the Brownian motion.
i.e. the Brownian motion $\beta(t)$ belongs to the Besov space $B^{1/2}_{p, q}$
if and only if $ q= \infty$ (with $1\leq p < \infty$.)
Then, they proved a variant of the bilinear estimate \eqref{KPVbilinear}
by Kenig-Ponce-Vega
adjusted to their Besov space setting,
establishing the local well-posedness via the fixed point theorem.
Note that the use of a variant of the bilinear estimate  \eqref{KPVbilinear}
required a slight regularization of the noise in space via $\phi$
so that the smoothed noise has the spatial regularity $s >-\frac{1}{2}$.
Thus, they could not treat the space-time white noise, i.e. $\phi =$ Id.

Our result is based on two observations.
The first one is that our $l^p_n$-based function spaces $\ft{b}^s_{p, \infty}$ in \eqref{Besov} and 
$X^{s, b}_{p, q}$ in \eqref{XSBP} capture the regularity of the spatial and space-time white noise
for $sp < -1$, $b  <  \frac{1}{2}$ and $q = 2$.
The second is that we can indeed carry out Bourgain's argument in \cite{BO3}, 
a nonlinear estimate on the second iteration, 
{\it without} assuming the a priori bound \eqref{BOO},
if we take the initial data $u_0 \in \ft{b}^s_{p, \infty}$ for  $s >-\frac{1}{2}$
with $p > 2$.
Then, we construct a solution $u$ as a strong limit of the smooth solutions $u^N$
(with smooth $u_0^N$ and $\phi^N$) of \eqref{SKDV1}.
Note that our nonlinear estimate on the second iteration in Section 5 depends 
on the stochastic term,
whereas the bilinear estimate in \cite{DDT1} is entirely deterministic.
 
Finally, we present our main results.
\begin{maintheorem} \label{THM1}
Let $\phi$ be as in \eqref{PHI}
and $p = 2+$.
Then, let $ s= -\frac{1}{2} + \dl$ with $\frac{p-2}{4p} < \dl < \frac{p-2}{2p}$. 
 i.e.  $sp < -1$.
Also, let $u_0$ be $\mathcal{F}_0$-measurable 
such that it has mean 0 and belongs to $\ft{b}^s_{p, \infty}(\mathbb{T})$ almost surely.
Then, there exists a stopping time $T_\omega >0$
and a unique process $u \in C([0, T_\omega]; \ft{b}^s_{p, \infty}(\mathbb{T}))$
satisfying \eqref{SKDV1} on $[0, T_\omega]$ almost surely.
\end{maintheorem}

As a corollary, we obtain the following:
\begin{maintheorem} \label{THM2}
The stochastic KdV \eqref{SKDV} with the additive space-time white noise 
is locally well-posed almost surely
(with the prescribed mean on $u_0$.)
\end{maintheorem}

\begin{remark}\rm
Our argument provides an answer to the question posed by Bourgain 
in \cite[Remark on p.120]{BO3}, 
at least in the local-in-time setting.
The deterministic part of the nonlinear estimate in Section 5
can be used to establish the local well-posedness of \eqref{KDV}
for  a finite Borel measure  $u_0 = \mu \in M(\T)$
with $\|\mu\| < \infty$ {\it without} the complete integrability or the smallness assumption on $\mu$.
Note that $\mu \in \ft{b}^s_{p, \infty}$
for $sp \leq -1$ since $\sup_n |\ft{\mu}(n)| <\|\mu\|< \infty.$
Hence, it can be used to  study the Cauchy problem on $M(\T)$ for non-integrable KdV-variants.
Also, see \cite{OHRIMS}.

\end{remark}

\begin{remark} \rm
Let $\mathcal{F} L^{s, p} (\mathbb{T})$ be the space of functions on $\mathbb{T}$
defined via the norm $\|f\|_{\mathcal{F} L^{s, p}} = \|\jb{n}^s \ft{f}(n)\|_{L^p_n}$. 
Recall from  \cite{OH3} that $\mathcal{F} L^{s, p} (\mathbb{T})$ contains the  support
of the (spatial) white noise when $sp < -1$.
Then, Theorems \ref{THM1} and \ref{THM2} can also be established in 
$\mathcal{F} L^{s, p} (\mathbb{T})$ for $s = -\frac{1}{2}+$, $p = 2+$ with $sp < -1$.
The modification is straightforward once we note that 
$\|f\|_{\mathcal{F} L^{s-\eps, p}} \lesssim  \|f\|_{\ft{b}^s_{p, \infty}}$ for any $\eps > 0$, 
and thus we omit the details.
\end{remark}

This paper is organized as follows:
In Section 2, we introduce some notations.
In Section 3, we introduce function spaces along with their embeddings 
and state deterministic linear estimates from \cite{BO1} and \cite{OH3}.
In Section 4, we study some basic properties of the stochastic convolution.
In Section 5, we prove Theorem \ref{THM1}
by establishing the nonlinear estimate on the second iteration 
of the integral formulation \eqref{duhamel1}.

\smallskip

\noindent
{\bf Acknowledgments:} 
The author would like to thank Prof. Jeremy Quastel 
and Prof. Catherine Sulem
for suggesting this problem.

\section{Notation}

In the periodic setting on $\T$, the spatial Fourier domain is $\Z$.
Let $dn$ be the normalized counting measure on $\Z$. 
We say $f \in L^p(\Z)$, $1 \leq p < \infty$, if
\[ \| f \|_{L^p(\mathbb{Z})} = \bigg( \int_{\mathbb{Z}} |f(n)|^p dn \bigg)^\frac{1}{p}  
:= \bigg( \frac{1}{2\pi} \sum_{n \in \mathbb{Z}} |f(n)|^p \bigg)^\frac{1}{p} < \infty.\]

\noindent
If $ p = \infty$, we have the obvious definition involving the supremum.
We often drop $2\pi$ for simplicity.
If a function depends on both $x$ and $t$, we use ${}^{\wedge_x}$ 
(and ${}^{\wedge_t}$) to denote the spatial (and temporal) Fourier transform, respectively.
However, when there is no confusion, we simply use ${}^\wedge$ to denote the spatial Fourier transform,
the temporal Fourier transform, and  the space-time Fourier transform, depending on the context.

For a Banach space $X \subset \mathcal{S}'(\mathbb{T} \times \mathbb{R})$, 
we use $\ft{X}$ to denote the space of the Fourier transforms of the functions in $X$,
which is a Banach space with the norm $\|f\|_{\ft{X}} = \| \mathcal{F}^{-1}_{n, \tau} f\|_{X}$,
where $\mathcal{F}^{-1}$ denotes the inverse Fourier transform (in $n$ and $\tau$.)
Also, for a space $Y$ of functions on $\mathbb{Z}$, 
we use $\ft{Y}$ to denote the space of the inverse Fourier transforms of the functions in $Y$
 with the norm $\|f\|_{\ft{Y}} = \| \mathcal{F} f\|_{Y}$.
Now, define $\ft{b}^s_{p, q}(\mathbb{T})$ by the norm 
\begin{align} \label{Besov2}
\| f\|_{\ft{b}^s_{p, q}(\mathbb{T})} 
= \| \ft{f}\|_{b^s_{p, q}(\mathbb{Z})} 
: = \big\| \| \jb{n}^s \ft{f}(n) \|_{L^p_{|n|\sim 2^j}} \big\|_{l^q_j} 
= \Big( \sum_{j = 0}^\infty\Big( \sum_{|n| \sim 2^j} \jb{n}^{sp} |\ft{f}(n)|^p 
 \Big)^\frac{q}{p} \Big)^\frac{1}{q}
\end{align}

\noi 
for $ q < \infty$ and by \eqref{Besov} when $q = \infty$.

Throughout the paper, $\eta(t)$ denotes a smooth cutoff supported on $[-1, 2]$ with $\eta(t) \equiv 1$ on $[0, 1]$,
and let $\eta_{_T}(t) = \eta(T^{-1}t)$.
We use $c,$ $ C$ to denote various constants, usually depending only on $s$, $p $,  and $\delta$.
If a constant depends on other quantities, we  make it explicit.
We use $A\lesssim B$ to denote an estimate of the form $A\leq CB$.
Similarly, we use $A\sim B$ to denote $A\lesssim B$ and $B\lesssim A$
and use $A\ll B$ when there is no general constant $C$ such that $B \leq CA$.
We also use $a+$ (and $a-$) to denote $a + \eps$ (and $a - \eps$), respectively,  
for arbitrarily small $\eps \ll 1$.

\section{Function Spaces and Basic  Embeddings}

First, let $X^{s, b}$ denote the usual periodic Bourgain space defined in \eqref{Xsb}.
We often use the shorthand notation $\|\cdot\|_{s, b}$ to denote the $X^{s, b}$ norm.
Now, define $X^{s, b}_{p, q}$, the Bourgain space adapted to $\ft{b}^s_{p, \infty}$, 
 to be the completion of the Schwartz functions on $\T\times \R$
with respect to the norm given by
\begin{equation} \label{XSBPQ}
 \| u \|_{X^{s, b}_{p, q}} 
 =  \|\jb{n}^s \jb{\tau - n^3}^b \ft{u}(n, \tau)\|_{b^0_{p, \infty} L^q_\tau}
 = \sup_j \|\jb{n}^s  \jb{\tau - n^3}^b \ft{u}(n, \tau)
\|_{L^p_{|n| \sim 2^j} L^q_\tau}.
\end{equation}

\noi
In the following, we take $p = 2+$ and $s = -\frac{1}{2}+ = -\frac{1}{2}+\dl$  
with $\dl < \frac{p-2}{2p}$ (and $\dl > \frac{p-2}{4p}$)
such that $sp < -1$.
Lastly, given $T> 0$, we define $X^{s, b, T}_{p, q}$ as a restriction of $X^{s, b}_{p, q}$ on $[0, T]$
by
\[ \|u\|_{X^{s, b, T}_{p, q}} 
=  \|u\|_{X^{s, b}_{p, q}[0, T]} 
= \inf \big\{ \|\wt{u} \|_{X^{s, b}_{p, q}}: {\wt{u}|_{[0, T]} = u}\big\}.\]

\noi
We define the local-in-time versions of the other function spaces analogously.

Now, we discuss the basic embeddings.
For $p \geq 2$, we have $\| a_n\|_{L^p_n} \leq \| a_n\|_{L^2_n}$.
Thus, we have $\|f\|_{\ft{b}^s_{p, \infty}} \leq \|f\|_{H^{s}}$, and thus
\begin{equation}\label{EMBED1}
\|u\|_{X^{s, b}_{p, 2}} \leq \|u\|_{X^{s, b}}.
\end{equation}

\noi
By H\"older inequality, we have
\begin{align} \label{EMBED2}
\|f\|_{H^{-\frac{1}{2}-\dl}} 
& = \Big(\sum_j (2^j)^{0-}
 \|\jb{n}^{-\frac{1}{2}-\dl+} \ft{f}(n) \|_{|n|\sim 2^j}^2  \Big)^\frac{1}{2} \notag\\
& \leq \sup_j \| \jb{n}^{-2\dl+} \|_{L^\frac{2p}{p-2}} \|\jb{n}^{-\frac{1}{2} + \dl} \ft{f}(n) \|_{L^p_n}
\leq \| f \|_{\ft{b}^s_{p, \infty}} 
\end{align}

\noi
for $s = -\frac{1}{2} +\dl$ with $\dl > \frac{p-2}{4p}$.
Hence, for $s = - \frac{1}{2} +\dl$ with $\dl > \frac{p-2}{4p}$, 
we have
\begin{equation} \label{EMBED3}
\|u\|_{X^{-\frac{1}{2}-\dl, b}}\lesssim \|u\|_{X_{p, 2}^{s, b}}.
\end{equation}


Now, we briefly go over the linear  estimates.
Let $S(t) = e^{-t \dx^3}$ and $T\leq 1$ in the following.
We first present the homogeneous and nonhomogeneous linear estimates.
See \cite{BO1}, \cite{KPV3}, \cite{OH3} for  details of the proofs.

\begin{lemma} \label{LEM:linear1}
For any $s \in \mathbb{R}$ and $b < \frac{1}{2}$, we have 
$\|  S(t) u_0\|_{X^{s, b, T}_{p, 2}} \lesssim T^{\frac{1}{2}-b}\|u_0\|_{\ft{b}^s_{p, \infty}}$.
\end{lemma}

\begin{lemma} \label{LEM:linear2}
For any $s \in \mathbb{R}$ and $b \leq \frac{1}{2}$, we have 
\begin{align*} 
 \bigg\|  \int_0^t S(t-t') F(x, t') dt'\bigg\|_{X^{s, b, T}_{p, 2}} 
\lesssim \| F \|_{X^{s, b-1}_{p, 2}} + \| F \|_{X^{s, -1}_{p, 1}}.
\end{align*}

\noi
Also, we have 
$ \big\|  \int_0^t S(t-t') F(x, t') dt'\big\|_{X^{s, b, T}_{p, 2}} 
\lesssim \| F \|_{X^{s, b-1}_{p, 2}}$
for $b > \frac{1}{2}$.
\end{lemma}

\noi
The next lemma is the periodic $L^4$ Strichartz estimate due to Bourgain \cite{BO1}.
\begin{lemma} \label{LEM:L4}
Let $u$ be a function on $\T \times \R$.
Then, we have 
$ \|u\|_{L^4_{x, t}} \lesssim \|u\|_{X^{0, \frac{1}{3}}}.$
\end{lemma}

\noi
Lastly, recall that by restricting the Bourgain spaces onto a small time interval
$[0, T]$, we can gain a small power of $T$. See Colliander-Oh \cite{CO1} for the proof.

\begin{lemma} \label{LEM:timedecay}
For $0\leq b' < b \leq \frac{1}{2}$, we have
\begin{equation*} 
\|u\|_{X^{s, b', T}} =\|\eta_{_T}u\|_{X^{s, b', T}} \lesssim T^{b-b'-} \|u\|_{X^{s, b}}.
\end{equation*}

\end{lemma}

\section{Stochastic Convolution}

In this section, we study basic properties of the stochastic convolution $\Phi(t)$ defined in \eqref{stoconv}.
In particular, we prove that 
$\eta \Phi $ belongs to $ X_{p, 2}^{s , b, T}$ 
and is continuous from $[0, T]$ into $ \ft{b}^s_{p, \infty}$ for $T \leq 1$ almost surely
for $sp < -1$ and $(b-1)\cdot2<-1$,
where  $\eta(t)$ is a smooth cutoff supported on $[-1, 2]$ with $\eta(t) \equiv 1$ on $[0, 1]$.

Before stating the main results, 
we point out the following.
Let $\phi$ be the identity operator on $L^2(\mathbb{T})$ or be as in \eqref{PHI}.
Then, we know that such $\phi$ is Hilbert-Schmidt from $L^2(\mathbb{T})$
into $H^s(\mathbb{T})$ if and only if $s < -\frac{1}{2}$.
In other words, with a slight abuse of notation, define
\begin{equation} \label{phiphi}
\phi := \sum_{n \in \mathbb{Z}} \phi e_n 
= \sum_{n \in \mathbb{Z}} \phi_n e_n
\end{equation}

\noi
in view of $\phi = \text{diag} (\phi_n ; n \ne 0)$.
Then, we have $\phi \in H^s(\mathbb{T})$ if and only if $s < -\frac{1}{2}$.
Moreover, we have 
$\|\phi\|_{{HS}(L^2; H^s)} = \|\phi\|_{H^s}$,
where $\|\cdot\|_{{HS}(L^2; H^s)}$ denotes the Hilbert-Schmidt norm 
from $L^2(\T)$ to $H^s(\T)$.
For such $\phi$, we also have $\phi \in \ft{b}^s_{p, \infty}(\mathbb{T})$ if and only if $s p \leq  -1$,
and we can use $\|\phi\|_{\ft{b}^s_{p, \infty}}$ 
to discuss the regularity of $\phi$
in place of the Hilbert-Schmidt norm.
This is one of the reasons for using this space.
(We need only $sp < -1$ for our purpose since the nonlinear estimate in Section 5 holds 
for $s = -\frac{1}{2}$ and $ p = 2+$ with $sp < -1$.)

\begin{proposition} \label{PROP:stoint}
Let $0 < T \leq 1$ and $p = 2+$.
Moreover, let $ s= -\frac{1}{2} + \dl$ and $b = \frac{1}{2} -\dl$ with 
$\frac{p-2}{4p} < \delta  < \frac{p-2}{2p}$.
i.e. $sp < -1$ and $(b-1)\cdot 2<-1$.
Then, for the stochastic convolution $\Phi(t)$ defined in \eqref{stoconv} 
with $\phi$ as in \eqref{PHI}, 
we have
\begin{equation} \label{stoint}
\mathbb{E} \big[\|\eta \Phi \|_{X_{p, 2}^{s, b, T}} \big] 
\leq C(\eta, s, p) < \infty.
\end{equation}

\noi
In particular, $\Phi \in X_{p, 2}^{-\frac{1}{2}+\dl, \frac{1}{2}-\dl, T} $ almost surely.
\end{proposition}

Before going into the proof of Proposition \ref{PROP:stoint}, recall the following.
Let $\beta_1$ and $\beta_2$ be independent real-valued Brownian motions on $(\Omega, \mathcal{F}, P)$, and
$f_1(t, \omega)$ and $f_2(t, \omega)$ be real-valued stochastic processes independent of $\beta_1$ and $\beta_2$.
Then, we can regard $\beta_j$ and $f_j$ 
as $\beta_j (t, \omega) = \beta_j (t, \omega_1)$ and $f_j(t, \omega) = f_j(t, \omega_2)$,
where $\omega = (\omega_1, \omega_2) \in \Omega_1 \times \Omega_2 = \Omega$.
Thus, in taking an expectation, we can first integrate over $\omega_1 \in \Omega_1$. 
Then, for $m \in \mathbb{N}$, we have
\begin{align} \label{Wienerestimate2}
\mathbb{E} & \bigg(  \Big|\int_a^b f_1(t) d \beta_1(t) + \int_a^b f_2(t) d\beta_2(t)\Big|^{2m}\bigg) \notag \\
&= \mathbb{E} \bigg(\sum_{k = 0}^{2m} 
\begin{pmatrix} 2m \\k \end{pmatrix}
\Big(\int_a^b f_1(t) d\beta_1(t)\Big)^k  \Big(\int_a^b f_2(t) d\beta_2(t)\Big)^{2m-k}\bigg) \notag \\
&= \mathbb{E}_{\Omega_2} \bigg[\sum_{n = 0}^{m}  \begin{pmatrix} 2m \\2n \end{pmatrix}
\frac{(2n)!}{2^n n!}\|f_1(\cdot, \omega_2)\|_{L^2(a, b)}^{2n} \frac{(2(m- n))!}{2^{m-n} (m-n)!} \|f_2(\cdot, \omega_2)\|_{L^2(a, b)}^{2(m-n)}\bigg].
\end{align}

\noi
In the computation above, we used the fact that, for each fixed $\omega_2$, 
$\int_a^b f_j(t, \omega_2) d \beta_j(t, \omega_1)$
is a Gaussian random variable on $\Omega_1$ with variance $\|f_j(\cdot, \omega_2)\|^2_{L^2(a, b)}$.

\begin{proof}
By H\"older inequality, we have
\begin{align*}
\| \jb{\tau - n^3}^{\frac{1}{2}-\dl} \ft{u}(n, \tau)\|_{L^2_\tau}
& \leq \|\jb{\tau - n^3}^{-2\dl} \|_{L_\tau^\frac{2p}{p-2}} 
\|\jb{\tau - n^3}^{\frac{1}{2}+\dl} \ft{u}(n, \tau)\|_{L^p_\tau}.
\end{align*}

\noi
i.e. We have 
$\|\eta \Phi \|_{X_{p, 2}^{s, \frac{1}{2}-\dl}}
\lesssim \|\eta \Phi \|_{X_{p, p}^{s, \frac{1}{2}+\dl}}$
as long as $\dl > \frac{p-2}{4p}$.
Thus,  we will work in  $X_{p, p}^{s, \frac{1}{2}+\dl}$
in the following.

Let $g(t) = \eta(t) \int_0^t S(-r) \phi(r) dW(r)$. 
i.e. $\eta(t) \Phi(\cdot, t) = S(t) g(\cdot, t)$. 
Assume that each $\beta_n$ is extended to a Brownian motion on $\R$
in such a way that the family $\{\beta_n\}_{n \geq 0}$ is still independent.
Note that for $t \in [0, T]$, we have 
\begin{equation} \label{GNT}
\ft{g}(n, t) = \eta(t) \int_0^t \eta(r) e^{-i r n^3}\phi_n(r)\chi_{[0, T]}(r) \tfrac{1}{\sqrt{2}}d\beta_n(r).
\end{equation}

\noindent
We have inserted $\eta(r)$ and $\chi_{[0, T]}(r)$ in the integrand 
since $\eta(r) \chi_{[0, T]}(r) \equiv 1$ for $r \in [0, t] \subset [0, T]$.
For notational simplicity, we use $\phi_n(r)$ to denote $\phi_n(r)\chi_{[0, T]}(r)$ in the following.
i.e. we assume that $\phi_n$ is supported on $[0, T]$.
By \eqref{PHI}, we have  $|\phi_n(r)| \leq 1$ for $r \in \mathbb{R}$.

Now, we write 
the left hand side of \eqref{stoint}  as 
\begin{align} \label{Westimate1}
\mathbb{E} \Big( \| \eta \Phi \|_{X_{p, p}^{s,\frac{1}{2}+\dl, T}} \Big)
\lesssim & \  \mathbb{E} \bigg[\sup_j 2^{js} 
\Big( \sum_{|n|\sim 2^j} \sum_{k = 1}^\infty 2^{kp(\frac{1}{2}+ \delta)} 
 \int_{|\tau|\sim 2^k} |\ft{g}(n, \tau)|^p d\tau \Big)^\frac{1}{p} \bigg] \notag \\
& + \mathbb{E} \bigg[\sup_j 2^{js} 
\Big( \sum_{|n|\sim 2^j}  \int_{|\tau|\leq 2} |\ft{g}(n, \tau)|^p d\tau\Big)^\frac{1}{p} \bigg].
\end{align}

\noindent
$\bullet$ {\bf Part 1:} 
First, we estimate the second term in \eqref{Westimate1}.
Let 
\begin{equation} \label{G_n}
G_n(r, \tau) = \eta(r) e^{-irn^3} \phi_n(r) \int_r^\infty \eta(t) e^{-it\tau}  dt.
\end{equation}

\noindent
Also write $\beta_n = \beta_n^{(r)} + i \beta_n^{(i)}$ 
where $\beta_n^{(r)} = \text{Re}\, \beta_n$ and $\beta_n^{(i)} = \text{Im}\, \beta_n $.
Then, by the stochastic Fubini Theorem,  we have, for $m \in \mathbb{N}$,
\begin{align} \label{W1estimate2}
\mathbb{E}\big[  |\ft{g}(n, \tau)|^{2m} \big]
& = \mathbb{E}\bigg( \Big|\int_\R \eta(t) e^{-it\tau}\int_{-\infty}^t \eta(r) e^{-irn^3} 
\phi_n(r) \tfrac{1}{\sqrt{2}} d \beta_n(r) dt\Big|^{2m} \bigg) \notag \\
&= 2^{-m} \mathbb{E}\bigg(\Big|\int_{-1}^2  G_n(r, \tau)  d \beta_n(r)\Big|^{2m}\bigg) \\
& \lesssim \mathbb{E} \bigg( \Big| \int_{-1}^2 \text{Re}G_n(r, \tau)  d \beta_n^{(r)}(r) 
-\int_{-1}^2 \text{Im}G_n(r, \tau)  d \beta_n^{(i)}(r) \Big|^{2m} \bigg) \notag \\
& \hphantom{X} + \mathbb{E} \bigg( \Big| \int_{-1}^2 \text{Im}G_n(r, \tau)  d \beta_n^{(r)}(r) 
+ \int_{-1}^2 \text{Re}G_n(r, \tau)  d \beta_n^{(i)}(r)\Big|^{2m} \bigg).
\notag 
\end{align}

\noi
Note that $|\text{Re} G_n(r, \tau)|, \, |\text{Im} G_n(r, \tau)| 
\leq |G_n(r, \tau)| \leq \| \eta\|_{L^1} |\phi_n(r)| \lesssim \| \eta\|_{L^1} \chi_{[0, T]}(r)$.
Thus, we have $\| \text{Re} G_n(r, \tau)\|_{L^2_r}^{2k}
\|\text{Im} G_n(r, \tau)\|_{L^2_r}^{2(m-k)}\lesssim \| \eta\|_{L^1}^{2m}$
for $k = 0, \cdots, m$.
Then, by \eqref{Wienerestimate2} along with the independence of $\phi_n$, $\beta_n^{(r)}$ and $\beta_n^{(i)}$, we have
\begin{align*}
\|\ft{g}(n, \tau)\|_{L^{2m}(\Omega)} \leq C = C(\eta, m)
\end{align*}

\noi
independent of $n$ and $\tau$.
Hence, for $p \in (2, 4)$,  we have
\begin{align} \label{XX1}
\big(\mathbb{E}\big[ |\ft{g}(n, \tau)|^p \big] \big)^\frac{1}{p}
\leq \|\ft{g}(n, \tau)\|^\theta_{L^2(\Omega)} \|\ft{g}(n, \tau)\|^{1-\theta}_{L^4(\Omega)}
\lesssim 1, 
\end{align}

\noi
by interpolation,
where $\theta \in (0, 1)$ such that 
$\frac{1}{p} = \frac{\theta}{2} + \frac{1-\theta}{4}$.
Then, the second term in \eqref{Westimate1} is estimated by
\begin{align} \label{XX2}
\eqref{Westimate1} & \leq \bigg( \sum_{j = 0}^\infty 2^{jsp}    \sum_{|n|\sim 2^j}
\int_{|\tau|\leq 2} \mathbb{E} \big[|\ft{g}(n, \tau)|^p\big] d\tau \bigg)^\frac{1}{p}
\lesssim \Big( \sum_{j = 0}^\infty 2^{jsp}    \sum_{|n|\sim 2^j} 1 
 \Big)^\frac{1}{p} \\
& \sim \Big( \sum_{j = 0}^\infty 2^{(sp + 1) j}    \Big)^\frac{1}{p}
\leq C < \infty, \notag
\end{align}

\noi
since $sp < -1$.

\noindent
$\bullet$ {\bf Part 2:} 
Next, we estimate the first term in \eqref{Westimate1}.
Let \begin{equation} \label{G12}
\begin{cases}
G^{(1)}_n(r, \tau) = \eta(r) e^{-i r n^3} \phi_n(r) \int_r^\infty \eta'(t) \frac{e^{-i t \tau}}{i \tau} dt,\\
G^{(2)}_n(r, \tau) = \eta^2(r) e^{-i r n^3} \phi_n(r)  \frac{e^{-i r \tau}}{i \tau}.
\end{cases}
\end{equation}

\noindent
Then,  by the stochastic Fubini theorem and integration by parts, we have 
\begin{align} \label{GNT2}
\sqrt{2} \ft{g}(n, \tau) & = \int_{-1}^2  G_n(r, \tau)  d \beta_n(r)  = \int_{-1}^2 G^{(1)}_n(r, \tau) d\beta_n(r)  
+ \int_{-1}^2  G^{(2)}_n(r, \tau) d\beta_n(r) \\
& =:  I^{(1)}_n(\tau) + I^{(2)}_n(\tau). \notag
\end{align}

\noi
Thus, we have $|\ft{g}(n, \tau)|^p \lesssim 
\big|I^{(1)}_n(\tau)\big|^p + \big|I^{(2)}_n(\tau)\big|^p$.

First, we estimate the contribution from $G^{(1)}_n$. 
For $|\tau|\sim 2^k$,  we have 
\begin{equation} \label{Partialinteg}
\bigg|\int_r^\infty \eta'(t) \frac{e^{-i t \tau}}{i \tau} dt \bigg| 
\leq | \tau^{-2}\eta'(r)| + \bigg|\int_r^\infty \eta''(t) \frac{e^{-i t \tau}}{\tau^2} dt \bigg|
\leq C_\eta 2^{-2k},
\end{equation}

\noi
by partial integration.
Thus, we have $|G_n^{(1)}(r, \tau)| \lesssim 2^{-2k}$.
Then, repeating a similar computation as in Part 1, we obtain
\begin{align} \label{XX3}
\big(\mathbb{E}\big[ |I^{(1)}_n(\tau)|^p \big] \big)^\frac{1}{p}
\leq \|I^{(1)}_n(\tau)\|^\theta_{L^2(\Omega)} \|I^{(1)}_n(\tau)\|^{1-\theta}_{L^4(\Omega)}
\lesssim 2^{-2k}, 
\end{align}

\noi
by \eqref{Wienerestimate2} and interpolation.
Hence, the contribution to \eqref{Westimate1} is estimated by
\begin{align} \label{XX4}
\eqref{Westimate1}
& \leq   
\bigg( \sum_{j=0}^\infty 2^{jsp} \sum_{|n|\sim 2^j} \sum_{k = 1}^\infty 2^{kp(\frac{1}{2}+ \delta)} 
 \int_{|\tau|\sim 2^k} \mathbb{E} \big[|I^{(1)}_n(\tau)|^p\big] d\tau \bigg)^\frac{1}{p}   \\
& \lesssim   
\Big( \sum_{j=0}^\infty 2^{j(sp+1)}  \sum_{k = 1}^\infty 
2^{k(-\frac{3p}{2}+ \delta p + 1)} \Big)^\frac{1}{p}  \leq C < \infty,  \notag
\end{align}

\noi
since $sp < -1$ and $-\frac{3p}{2}+ \delta p + 1 < 0$.

\medskip

Now, we consider the contribution from $I^{(2)}_n(\tau)$.
With $\beta_n = \beta_n^{(r)} + i \beta_n^{(i)}$, we have 
$|I^{(2)}_n(\tau)|^2 \lesssim \Big|\int_{-1}^2  G^{(2)}_n(r, \tau) d\beta^{(r)}_n(r)\Big|^2 
+ \Big|\int_{-1}^2  G^{(2)}_n(r, \tau) d\beta^{(i)}_n(r)\Big|^2$.
We only estimate the first term since the second term is estimated in the same way.
By Ito formula (c.f. \cite{DDT1}), we have
\begin{align*}
\bigg|\int_{-1}^2  G^{(2)}_n(r, \tau)&  d\beta^{(r)}_n(r)\bigg|^2
=  \int_{-1}^2 \eta^4(t) \frac{|\phi_n(t)|^2}{\tau^2} dt \\
& + 2 \text{Re} \int_{-1}^2 \int_{-\infty}^t G_n^{(2)} (r, \tau) d \beta_n^{(r)}(r) 
\cj{G_n^{(2)}} (t, \tau) d \beta_n^{(r)} (t)
=: I_n'(\tau) + I_n''(\tau).
\end{align*}

\noindent
The contribution from $I'_n(\tau)$ is at most
\begin{align} \label{XX5}
\eqref{Westimate1}
& \lesssim   
\bigg( \sum_{j=0}^\infty 2^{jsp} \sum_{|n|\sim 2^j} \sum_{k = 1}^\infty 2^{kp(\frac{1}{2}+ \delta)} 
 \int_{|\tau|\sim 2^k} |\tau|^{-p} d\tau  
 \Big(\int_{-1}^2 \eta^4(t) dt \Big)^\frac{p}{2}\bigg)^\frac{1}{p}  \\ 
& \lesssim  \|\eta\|_{L^4}^2 
\Big( \sum_{j=0}^\infty 2^{j(sp+1)} \sum_{k = 1}^\infty 2^{k (-\frac{p}{2}+ \delta p + 1)} 
 \Big)^\frac{1}{p}  
 \leq C < \infty, \notag
\end{align}

\noi
since $sp < -1$ and $\delta < \frac{p-2}{2p}$.
 
We finally estimate the  contribution from $I''_n(\tau)$.
Write  $ I''_n(\tau) =  \int_{-1}^2 H_n(t)  d \beta_n^{(r)} (t)$,
where $H_n(t) = \int_{-\infty}^t  \wt{H}_n (r, t)  d \beta_n^{(r)} (r)$ with 
\begin{equation} \label{H_n}
\wt{H}_n (r, t) =  2 \tau^{-2} \text{Re} \big(\eta^2(r) \eta^2(t) e^{i(t-r)n^3} \phi_n(r) \cj{\phi_n(t) } 
 e^{i (t - r) \tau} \big).
\end{equation}

\noi
Then, by Ito isometry and $|\phi_n (w, t)| \leq 1$ for all $(\omega, t) \in \Omega \times \mathbb{R}$, we have
\begin{align} \label{XX6}  
\mathbb{E} \big[ |I''_n(\tau)|^2 \big] &= 
\mathbb{E} \bigg[  \Big(\int_{-1}^2 H_n(t)  d \beta_n^{(r)} (t) \Big)^{2} \bigg] 
\sim \int_{-1}^2 \mathbb{E}\big[ H_n^{2}(t)\big] dt \notag \\
&= \int_{-1}^2 \mathbb{E}\Big[ \Big(\int_{-\infty}^t  \wt{H}_n (r, t)  d \beta_n^{(r)} (r)\Big)^2 \Big] dt 
= \int_{-1}^2 \int_{-1}^t  \mathbb{E}\big[ |\wt{H}_n (r, t)|^2  \big] dr dt \\
& \lesssim  \tau^{-4}  \int_{-1}^2 \int_{-1}^t  \eta^4(r) \eta^4(t)  dr dt 
\lesssim \tau^{-4}. \notag
\end{align}

\noi
Hence, the contribution from $I''_n(\tau)$ is at most
\begin{align} \label{XX7}
\eqref{Westimate1}
& \lesssim   
\bigg( \sum_{j=0}^\infty 2^{jsp} \sum_{|n|\sim 2^j} \sum_{k = 1}^\infty 2^{kp(\frac{1}{2}+ \delta)} 
 \int_{|\tau|\sim 2^k} \mathbb{E} \big[ |I''_n(\tau)|^\frac{p}{2} \big] d\tau  \bigg)^\frac{1}{p}   \notag \\ 
& \lesssim   
\bigg( \sum_{j=0}^\infty 2^{jsp} \sum_{|n|\sim 2^j} \sum_{k = 1}^\infty 2^{kp(\frac{1}{2}+ \delta)} 
 \int_{|\tau|\sim 2^k} \big(\mathbb{E} \big[ |I''_n(\tau)|^2 \big] \big)^\frac{p}{4} d\tau  \bigg)^\frac{1}{p}   \\ 
& \lesssim  
\Big( \sum_{j=0}^\infty 2^{j(sp+1)} \sum_{k = 1}^\infty 2^{k (-\frac{p}{2}+ \delta p + 1)} 
 \Big)^\frac{1}{p}  
 \leq C < \infty,   \notag
\end{align}

\noi
for $p \leq 4$,  $sp < -1$, and $\delta < \frac{p-2}{2p}$.
\end{proof}

We state a corollary to the proof of Proposition \ref{PROP:stoint}
for a general diagonal covariance operator 
$\phi (t, \omega)= \text{diag} (\phi_n (t, \omega) ; n \in \mathbb{Z})$,
which is independent of $\{\beta_n\}_{n \geq 1}$.

\begin{corollary} \label{COR:stoint}
Let $0 < T \leq 1$, $p = 2+$, and $s, s' \in \R$ with $s < s'$.
Moreover, let $b = \frac{1}{2} -\dl$ with 
$\frac{p-2}{4p} < \delta  < \frac{p-2}{2p}$.
i.e. $(b-1)\cdot 2<-1$.
Then, for the stochastic convolution $\Phi(t)$ defined in \eqref{stoconv} 
with $\phi \in L^p( [0, T] \times \Omega; \ft{b}^{s'}_{p, \infty})$,
independent of $\{\beta_n\}_{n \geq 1}$,  
we have
\begin{equation} \label{stoint2}
\mathbb{E} \big[\|\eta \Phi \|_{X_{p, 2}^{s, b, T}} \big] 
\leq C(\eta, s, s', p) \|\phi\|_{L^p([0, T] \times \Omega; \ft{b}^{s'}_{p, \infty})}.
\end{equation}

\noi
In particular, $\Phi \in X_{p, 2}^{s, \frac{1}{2}-\dl, T} $ almost surely.
\end{corollary}

\begin{proof}
In the proof of Proposition \ref{PROP:stoint}, 
we used $|\phi_n(t)| \leq 1$ whenever $\phi_n(t)$ appeared. 
Now, we briefly go through the proof of Proposition \ref{PROP:stoint},
keeping track of $\phi_n(t)$.
Since $\phi$ is independent of $\{\beta_n\}_{n \geq 1}$, 
we regard $\beta_n$ and $\phi_n$ 
as $\beta_n (t, \omega) = \beta_n (t, \omega_1)$ and $\phi_n(t, \omega) = \phi_n(t, \omega_2)$,
where $\omega = (\omega_1, \omega_2) \in \Omega_1 \times \Omega_2 = \Omega$.

In \eqref{XX1}, we have 
$ \mathbb{E}\big[ |\ft{g}(n, \tau)|^p \big] 
\lesssim \mathbb{E}_{\Omega_2}\|\phi_n( \cdot, \omega_2)\|^p_{L^2[0, T]}$.
Then, in \eqref{XX2}, we have
\begin{align*} 
\eqref{Westimate1} & \leq \bigg( \sum_{j = 0}^\infty 2^{jsp}    \sum_{|n|\sim 2^j}
\int_{|\tau|\leq 2} \mathbb{E}_{\Omega_2}\|\phi_n( \cdot, \omega_2)\|^p_{L^2[0, T]} d\tau \bigg)^\frac{1}{p}\\
& \leq
\bigg(\sum_{j = 0}^\infty 2^{j(s-s')p}2^{j s'p} 
\sum_{|n|\sim 2^j}
\|\phi_n( \cdot, \omega_2)\|^p_{L^p([0, T]\times \Omega_2)} \bigg)^\frac{1}{p} \\
& \lesssim  \|\phi\|_{L^p([0, T] \times \Omega; \ft{b}^{s'}_{p, \infty})}  
\end{align*}

\noi
since $s - s' < 0$.
A similar modification in \eqref{XX3} and \eqref{XX4} (and \eqref{XX5})
takes care of the contribution from $I^{(1)}_n(\tau)$ (and $I'_n(\tau)$, respectively.)
Now, as for $I''_n(\tau)$, we first integrate only over $\Omega_1$ in \eqref{XX6}
and obtain
\begin{align*}   
\mathbb{E}_{\Omega_1} \big[ |I''_n(\tau)|^2 \big] 
& \lesssim  \tau^{-4}  \int_{-1}^2 \int_{-1}^t  \eta^4(r) \eta^4(t) 
|\phi_n(r)|^2|\phi_n(t)|^2 dr dt 
\lesssim \tau^{-4} \|\phi_n\|_{L^2[0, T]}^4. 
\end{align*}

\noi
Then, in \eqref{XX7}, we have
\begin{align*}
\mathbb{E}\big[ |I''_n(\tau)|^\frac{p}{2} \big]
= \mathbb{E}_{\Omega_2} \big[ \|I''_n(\tau)\|_{L^\frac{p}{2}(\Omega_1)}^\frac{p}{2}\big]
\leq \mathbb{E}_{\Omega_2} \big[ \|I''_n(\tau)\|_{L^2(\Omega_1)}^\frac{p}{2}\big]
\lesssim \tau^{-p} \,\mathbb{E}_{\Omega_2} \|\phi_n(\cdot, \omega_2) \|_{L^2[0, T]}^p
\end{align*}

\noi
for $p \in [2, 4]$.  
The rest follows as before.
\end{proof}

Now, we discuss the continuity of the stochastic convolution.
In the remaining of this section, 
 we show that the stochastic convolution $\Phi(t)$ defined in \eqref{stoconv} belongs to 
$C([0, T]; \ft{b}^s_{p, \infty}(\T))$ almost surely.
With $\beta_n = \beta_n^{(r)} + i\beta_n^{(i)}$, we have 
\begin{equation} \label{W3estimate0}
\Phi(t) = \frac{1}{\sqrt{2}} \sum_{n \ne 0} \int_0^t S(t - r) \phi_n(r) e_n d \beta_n^{(r)}(r)
+ i \frac{1}{\sqrt{2}} \sum_{n \ne 0} \int_0^t  S(t - r) \phi_n (r)e_n d \beta_n^{(i)}(r),
\end{equation}

\noindent
since
$\phi e_0 = 0$ and  $\phi e_n = \phi_n e_n$, $n\ne 0$.
In the following, we only show the continuity of the  first stochastic convolution in \eqref{W3estimate0},
which we shall denote by $\Phi^{(r)} (t)$.
Also, let $W^{(r)} (t) = \frac{1}{\sqrt{2}}\sum_n \beta_n^{(r)}(t) e_n$.
As in Da Prato \cite{DAPRATO1}, we use the factorization method based on the elementary identity
\begin{equation} \label{W3estimate1}
\int_r^t (t-t')^{\al - 1} (t' - r)^{-\al} d t' = \frac{\pi}{\sin \pi \al},
\end{equation}

\noindent
with $\al \in (0, 1)$ for $0 \leq r \leq t' \leq t$. 
Using \eqref{W3estimate1}, we can write the first term in \eqref{W3estimate0} as 
\begin{equation} \label{W3estimate2}
\Phi^{(r)} (t) = \frac{\sin \pi \al}{\pi} \int_0^t S(t-t') (t-t')^{\al - 1} Y(t') d t',
\end{equation}

\noindent
where
\begin{equation} \label{W3estimate3}
Y(t') =  \int_0^{t'} S(t'-r) (t'-r)^{-\al} \phi(r) dW^{(r)}(r).
\end{equation}

First, we present the following lemma
which provides a criterion for the continuity of \eqref{W3estimate2}
in terms of the $L^{2m}$-integrability of $Y(t')$.

\begin{lemma} [Lemma 2.7 in  \cite{DAPRATO1}]\label{LEM:stoconti1}
Let $T>0$, $\al \in (0, 1)$, and $m > \frac{1}{2\al}$.
For $f \in L^{2m} ( [0, T]; \ft{b}^s_{p, \infty}(\T) )$, let
\begin{equation*}
F(t) = \int_0^t S(t-t') (t-t')^{\al -1} f(t') dt', \quad 0 \leq t \leq T.
\end{equation*}

\noindent
Then, $F \in C([0, T]; \ft{b}^s_{p, \infty}(\T))$.
Moreover, there exists $C = C(m, T)$ such that 
\begin{equation*}
\|F(t)\|_{\ft{b}^s_{p, \infty}} \leq  C\|f\|_{L^{2m} ( [0, T]; \ft{b}^s_{p, \infty} )}, \quad 0 \leq t \leq T.
\end{equation*}
\end{lemma}

\begin{remark} \rm
Although Lemma 2.7 in \cite{DAPRATO1} is stated for a Hilbert space $H$,
its proof makes no use of the Hilbert space structure of $H$. 
Thus the same result  holds for $\ft{b}^s_{p, \infty}(\T)$ as well.
\end{remark}

In view of Lemma \ref{LEM:stoconti1}, 
it suffices to show that $Y(t') \in L^{2m}([0, T]; \ft{b}^s_{p, \infty}(\T))$ a.s.

\begin{proposition} \label{PROP:stoconti}
Let $T > 0$, $m \geq 2$, $s = -\frac{1}{2}+$, and $ p = 2+$ such that $sp < -1$.
Let $\phi$ be as in \eqref{PHI}.
Then, the stochastic convolution $\Phi^{(r)}(t)$  is continuous from $[0, T]$ 
into $\ft{b}^s_{p, \infty}$ almost surely.
Moreover, there exists 
\[ \mathbb{E} \Big( \sup_{t \in [0, T]} \|\Phi^{(r)} (t) \|^{2m}_{\ft{b}^s_{p, \infty}} \Big) \leq
C(m, T, s, p) < \infty.\]

\end{proposition}

\begin{proof} 
Let $\al \in \big( \frac{1}{2m}, \frac{1}{2}\big)$ and $Y$ be as in \eqref{W3estimate3}.
First, note that $Y$ is real-valued since $ \phi_{-n}(s) e_{-n}  = \cj{\phi_n(s) e_n}$
and $\beta_{-n}^{(r)} = \beta_n^{(r)}$.
Note that $\{\beta^{(r)}_n\}_{n \ne 0}$ and $\phi$  are independent
since $\phi$ depends only on $\beta_0$.
Thus,  
we can regard $\beta^{(r)}_n$ and $\phi$ as $\beta^{(r)}_n(\omega) = \beta^{(r)}_n(\omega_1)$ and $\phi(\omega) = \phi(\omega_2)$, 
where $\omega = (\omega_1, \omega_2) \in \Omega_1 \times \Omega_2 = \Omega$.
Then, for each fixed $\omega_2$ and $t' \in [0, t]$, 
$\ft{Y(t')}(n)$ is a Gaussian random variable on $\Omega_1$
with $\text{Var}_{\Omega_1}\big(\ft{Y(t')}(n)\big) = 
\mathbb{E}_{\Omega_1} \big [|\ft{Y(t')}(n)|^2 \big]$. 

Let $G_n(r, \omega_2) =  (t'-r)^{-\al} e^{i(t'-r)n^3} \phi_n(r, \omega_2)$. 
Note that $|G_n(r, \omega_2)| = (t'-r)^{-\al}$ for $0 < r < t'$ and $n\ne 0$.
By Ito isometry, we have 
\begin{align*}
\mathbb{E}_{\Omega_1} \big [|\ft{Y(t')}(n)|^2 \big] 
&= \frac{1}{2} \, \mathbb{E}_{\Omega_1} \bigg[\Big|\int_0^{t'} G_n(r, \omega_2) d \beta(r, \omega_1)\Big|^2\bigg]\\
&=\frac{1}{2} \int_0^{t'} |G_n(r, \omega_2)|^2 dr \sim \int_0^{t'} (t' -r)^{-2\alpha} dr.
\end{align*}

\noi
Then, by Minkowski  integral inequality (with $p = 2+ < 2m$) after replacing $\sup_j$ by $\sum_j$, we have
\begin{align*}
\mathbb{E}_{\Omega_1}  \big( \|Y(t', \cdot, & \omega_2)\|_{\ft{b}^s_{p, \infty}}^{2m}\big)
= \mathbb{E}_{\Omega_1} \Big[ \Big( \sup_j \sum_{|n|\sim 2^j} \jb{n}^{sp}  
|\ft{Y(t')}(n)|^p \Big)^\frac{2m}{p}\Big] \\
& \lesssim  \bigg( \sum_{j=0}^\infty \sum_{|n|\sim 2^j} 2^{jsp}  
 \Big(\mathbb{E}_{\Omega_1} \big [|\ft{Y(t')}(n)|^{2m} \big]\Big)^\frac{p}{2m} \bigg)^\frac{2m}{p} \\
& \sim \Big( \sum_{j=0}^\infty  2^{j(sp+1)}  \Big)^\frac{2m}{p}  
\bigg(\int_0^{t'} (t'-r)^{-2\al}  dr\bigg)^m 
 \lesssim \bigg(\frac{(t')^{1-2\alpha}}{1-2\alpha}\bigg)^m,
\end{align*}

\noi
since $sp < -1$.
Therefore, we have
\begin{align*}
\int_0^T \mathbb{E} \big( & \| Y(t') \|_{\ft{b}^s_{p, \infty}}^{2m} \big) dt' = 
\int_0^T \mathbb{E}_{\Omega_2} \mathbb{E}_{\Omega_1} \big( \|Y(t')\|_{\ft{b}^s_{p, \infty}}^{2m} \big) dt'\\
& \lesssim  \int_0^T \bigg( \frac{(t')^{1-2\al}} {1- 2\al} \bigg)^m dt' 
 \lesssim T^{(1-2\al) m + 1} < C(m, T, s, p) < \infty.
\end{align*}

\noindent
In particular, it follows that $Y(\cdot, \omega) \in L^{2m} ([0, T]; \ft{b}^s_{p, \infty})$ almost surely.
Then, the desired result follows from Lemma \ref{LEM:stoconti1}.
\end{proof}

\section{Nonlinear Estimate on the Second Iteration}
Now, we present the crucial nonlinear analysis.
First, we briefly go over Bourgain's argument in \cite{BO3}.
By writing the integral equation, the deterministic KdV \eqref{KDV} is equivalent to
\begin{equation} \label{KDVduhamel}
u(t) =  S(t) u_0 -\tfrac{1}{2} \mathcal{N}(u, u) (t), 
\end{equation} 

\noi
where $\mathcal{N}(\cdot, \cdot)$  is  given by
\begin{equation} \label{NN}
\mathcal{N}(u_1, u_2) (t) :=  \int_0^t S(t - t') \dx (u_1 u_2)(t') d t'.
\end{equation}

In the following, we assume that the initial condition $u_0$ has the mean 0,
which implies that $u(t)$ has the spatial mean 0 for each $t\in \R$.
We use $(n, \tau)$, $(n_1, \tau_1)$, and $(n_2, \tau_2)$ to denote the Fourier variables
for $uu$, the first factor, and the second factor $u$ of $uu$ in $\mathcal{N}(u, u)$, 
respectively.
i.e. we have $n = n_1 + n_2$ and $\tau = \tau_1 + \tau_2$.
By the mean 0 assumption on $u$ and 
by the fact that we have $\dx (uu)$ in the definition of $\mathcal{N}(u,u)$, 
we assume  $n, n_1, n_2 \ne 0$.
We also use the following notation:
\[\s_0 := \jb{\tau - n^3} \text{ and } \s_j := \jb{\tau_j - n_j^3}.\]

\noi
One of the main ingredients is the observation due to Bourgain \cite{BO1}:
\begin{equation} \label{Walgebra}
n^3 - n_1^3 - n_2^3 = 3 n n_1 n_2, \ \text{for } n = n_1 + n_2,
\end{equation}

\noi
which in turn implies that 
\begin{equation} \label{MAXMAX}
\MAX:= \max( \s_0, \s_1, \s_2) \gtrsim \jb{n n_1 n_2}.
\end{equation}

Now, define 
\begin{equation} \label{AJJ}
A_j = \{(n, n_1, n_2, \tau, \tau_1, \tau_2) \in \mathbb{Z}^3 \times \R^3:
\s_j = \MAX\},
\end{equation}

\noi
and let $\mathcal{N}_j(u, u)$ denote the contribution of $\mathcal{N}(u, u)$ on $A_j$.
By the standard bilinear estimate as in \cite{BO1}, \cite{KPV4}, 
we have
\begin{align} \label{N_0}
\|\mathcal{N}_0(u, u)\|_{{-\frac{1}{2} + \dl, \frac{1}{2}-\dl}}
\leq o(1)\|u\|^2_{{-\frac{1}{2} - \dl, \frac{1}{2}-\dl}},
\end{align}

\noi
where $o(1) = T^\theta$ with some $\theta > 0$ by considering the estimate 
on a short time interval $[-T, T]$ (e.g. Lemma \ref{LEM:timedecay}).
See (2.17), (2.26), and (2.68) in \cite{BO3}.
Here, we abuse the notation and use $\|\cdot\|_{s, b} = \|\cdot\|_{X^{s, b}}$ 
to denote the local-in-time version as well.
Note that the temporal regularity $b = \frac{1}{2} - \dl < \frac{1}{2}$.
This allowed us to gain the spatial regularity by $2\dl$.
Clearly, we can not expect to do the same for $\mathcal{N}_1(u, u)$.
(By symmetry, we do not consider $\mathcal{N}_2(u, u)$ in the following.)
The bilinear estimate \eqref{KPVbilinear} is known to fail for any $s \in \R$
if $b < \frac{1}{2}$ due to the contribution from $\mathcal{N}_1(u, u)$. 
See \cite{KPV4}.
Following the notation in \cite{BO3}, 
let 
\begin{equation}\label{Isb}
I_{s, b} = \|\mathcal{N}_1(u, u) \|_{X^{s, b}} \
\text{ and }\ \al := \frac{1}{2} -\dl < \frac{1}{2}.
\end{equation}

\noi
Then, by Lemma \ref{LEM:linear2} and duality 
with $\|d(n, \tau)\|_{L^2_{n, \tau}} \leq 1$, we have 
\begin{align} \label{eq:I1}
I_{-\al, 1-\al} & = \|\mathcal{N}_1(u, u) \|_{-\al, 1-\al}\\
& \lesssim \sum_{\substack{n, n_1\\n = n_1 + n_2}} \intt_{\tau = \tau_1 + \tau_2} d\tau d\tau_1
\frac{\jb{n}^{1-\al}d(n, \tau)}{\s_0^{\al }} \ft{u}(n_1, \tau_1) \frac{\jb{n_2}^{1-\al}c(n_2, \tau_2)}{\s_2^\al},
\notag
\end{align}

\noi
where  
\begin{equation} \label{CN}
c(n_2, \tau_2) = \jb{n_2}^{-(1-\al)}\s_2^{\al}\, \ft{u}(n_2, \tau_2) 
\text{ so that }
\|c\|_{L^2_{n, \tau}} = \|u\|_{-(1-\al), \al} = \|u\|_{-\frac{1}{2}-\dl, \frac{1}{2}-\dl}.
\end{equation}

\noi
The main idea here is to consider the second iteration, 
i.e. substitute \eqref{KDVduhamel} for  $\ft{u}(n_1, \tau_1)$ in \eqref{eq:I1}, 
thus leading to a trilinear expression.
Since $\s_1 = \MAX \gtrsim \jb{n n_1n_2}\gg1$ on $A_1$, 
we can assume that 
\begin{equation} \label{eq:u_1}
\ft{u}(n_1, \tau_1) = \big(\mathcal{N}(u, u)\big)^\wedge(n_1, \tau_1)
\sim \frac{|n_1|}{\s_1} \sum_{n_1 = n_3 + n_4} \intt_{\tau_1 = \tau_3 + \tau_4}
\ft{u}(n_3, \tau_3)\ft{u}(n_4, \tau_4) d\tau_4.
\end{equation}

\noi
Note that $\ft{u}(n_1, \tau_1)$ can not come from $S(t) u_0$  of \eqref{KDVduhamel}
since we have $\s_1 \sim 1$ for the linear part.
Moreover, by the standard computation \cite{BO1}, we have
\begin{align} \label{Duhamel}
 \mathcal{N}(u, u)(x, t) 
& = -i \sum_{k= 1}^\infty \frac{i^k t^k}{k!}  
\sum_{n \ne 0} e^{i(nx + n^3t)} \int \eta(\ld - n^3) \ft{\dx u^2}(n, \ld) d\ld \notag \\
& \hphantom{X}+ i \sum_{n \ne 0} e^{inx} \int 
\frac{\big(1-\eta\big)(\tau - n^3)}{\tau - n^3}  
\ft{\dx u^2}(n, \tau) e^{i \tau t } d \tau \notag \\
& \hphantom{X}+ i  \sum_{n \ne 0} e^{i(nx + n^3 t)} \int 
\frac{\big(1-\eta\big)(\ld - n^3)}{\ld - n^3}  \ft{\dx u^2}(n, \ld) d \ld \notag \\
& =:  \mathcal{M}_1(u, u)(x, t) + \mathcal{M}_2(u, u)(x, t)+ \mathcal{M}_3(u, u)(x, t).
\end{align}

\noi
Note that 
$(\mathcal{M}_1(u, u))^\wedge(n_1, \tau_1)$
and $(\mathcal{M}_3(u, u))^\wedge(n_1, \tau_1)$
are distributions supported on $\{\tau_1 -n_1^3 = 0\}$.
i.e. $\s_1 \sim 1$.
Hence, the only contribution for the second iteration on $A_1$ comes from 
$\mathcal{M}_2(u, u)$
whose Fourier transform is given in 
\eqref{eq:u_1}.
This shows the validity of the assumption \eqref{eq:u_1}.

Note that the $\s_1 $ appearing in the denominator
allows us to cancel $\jb{n}^{1-\al}$ and $\jb{n_2}^{1-\al}$ in the numerator in \eqref{eq:I1}.
Then, $I_{-\al, 1-\al}$ can be estimated by 
\begin{align} \label{eq:I2}
\lesssim \sum_{\substack{n = n_1 + n_2\\n_1 = n_3 + n_4}} 
\intt_{\substack{\tau = \tau_1 + \tau_2\\\tau_1 = \tau_3 + \tau_4}} 
\frac{\jb{n}^{1-\al}d(n, \tau)}{\s_0^{\al }} \frac{|n_1|}{\s_1} \, \ft{u}(n_3, \tau_3)\ft{u}(n_4, \tau_4)
\frac{\jb{n_2}^{1-\al}c(n_2, \tau_2)}{\s_2^\al}
.
\end{align}

\noi
Then, Bourgain divided the argument into several cases, 
depending on the sizes of $\s_0, \cdots, \s_4$.
Here, the key algebraic relation is
\begin{equation}\label{algebra2}
n^3 - n_2^3 - n_3^3 - n_4^3 = 3(n_2+ n_3)(n_3+ n_4)(n_4+ n_2), \  \text{ with } n = n_2 + n_3 + n_4.
\end{equation}

\noi
Then, Bourgain proved -see (2.69) in \cite{BO3}-
\begin{equation} \label{N_1}
I_{-\al, 1-\al} \leq o(1)\|u\|_{-(1-\al), \al} I_{-\al, 1-\al}  
+ o(1) \|u\|^3_{-(1-\al), \al} + o(1) \|u\|_{-(1-\al), \al},
\end{equation}

\noi
{\it assuming} the a priori estimate \eqref{BOO}: $|\ft{u}(n, t)| <C$ for all $n\in \mathbb{Z}$, $t \in\R$.
Indeed, the estimates involving the first two terms on the right hand side of  \eqref{N_1}
were obtained without \eqref{BOO}, and
{\it only} the last term in \eqref{N_1} required \eqref{BOO}, 
-see ``Estimation of (2.62)'' in \cite{BO3}-, 
which was then used to deduce
\begin{equation} \label{eq:apriori}
\|\ft{u}(n, \cdot)\|_{L^2_\tau} < C.
\end{equation}

\noi
The a priori estimate \eqref{BOO} is derived via
the isospectral property of the KdV flow
and is false for a general function in $X^{-(1-\al), \al}$.
(It is here that the smallness of the total variation $\|\mu\|$ is used.)

\medskip

Our goal is to carry out a similar analysis for SKdV \eqref{SKDV1} on the second iteration 
{\it without} the a priori estimates \eqref{BOO} and \eqref{eq:apriori} coming from the complete integrability of KdV.
We achieve this goal by considering the estimate
in $X^{-\al, \al}_{p, 2} = X^{-\frac{1}{2}+ \dl, \frac{1}{2}-\dl}_{p, 2}$, 
where $p = 2+$ and $\frac{p-2}{4p} < \dl< \frac{p-2}{2p}$.
By \eqref{EMBED1} and \eqref{EMBED3}
(recall $-\al = -\frac{1}{2} +\dl$ and $-(1-\al) = -\frac{1}{2} -\dl$), we have
\begin{equation} \label{EMBED5}
\|u\|_{X^{-\al, \al}_{p, 2}} \leq \|u\|_{X^{-\al, \al}},
\text{ and } \
\|u\|_{X^{-(1-\al), \al}} \lesssim \|u\|_{X^{-\al, \al}_{p, 2}}.
\end{equation}

\noi
Then, it follows from \eqref{N_0} and \eqref{EMBED5} that 
\begin{equation} \label{Np0}
\|\mathcal{N}_0(u, u)\|_{X^{-\al, \al}_{p, 2}} \leq o(1) \|u\|^2_{X^{-\al, \al}_{p, 2}}.
\end{equation}

Now, we consider the estimate on $\|\mathcal{N}_1(u, u) \|_{X^{-\al, \al}_{p, 2}}$.
From \eqref{EMBED5} and $\al < 1-\al$, it suffices to control $I_{-\al, 1-\al}$.
As in the deterministic case, we consider the second iteration, 
and substitute \eqref{duhamel1} for  $\ft{u}(n_1, \tau_1)$ in \eqref{eq:I1}.
As before, there is no contribution from $S(t) u_0$, 
or $\mathcal{M}_1(u, u)$, $\mathcal{M}_3(u, u)$
defined  in \eqref{Duhamel}.
Now, there are two contributions:

\begin{itemize}
\item[(i)] $\mathcal{N}_1(\mathcal{M}_2(u, u), u)$ from the deterministic nonlinear part: 
In this case, we can use the estimates from \cite{BO3} 
{\it except} when the a priori bound \eqref{BOO} was assumed.
i.e. we need to estimate the contribution from (2.62) in \cite{BO3}:
\begin{equation}\label{262}
R_\al := 
\sum_{n} \intt_{\tau = \tau_2 +  \tau_3 + \tau_4} 
\chi_B \frac{d(n, \tau)}{\jb{n}^{1+\al}\s_0^{\al }} 
 \ft{u}(-n, \tau_2) \ft{u}(n, \tau_3)\ft{u}(n, \tau_4)d\tau_2 d\tau_3d\tau_4,
\end{equation}

\noi
where $\|d(n, \tau)\|_{L^2_{n, \tau}} \leq 1$
and $B = \{ \s_0, \s_2, \s_3, \s_4 < |n|^\g\}$ with some small parameter $\g>0$. 
Note that this corresponds to the case $n_2 = -n$ and $n_3 = n_4 = n$
in \eqref{eq:I2} after some reduction.
In our analysis, we directly estimate $R_\al$ 
in terms of $\|u\|_{X^{-\al, \al}_{p, 2}}$.
The key observation is that 
we can take the spatial regularity $s = -\al $ to be greater than $-\frac{1}{2}$ 
by choosing $p > 2$.

\item[(ii)] $\mathcal{N}_1(\Phi, u)$ from the stochastic convolution $\Phi$ in \eqref{stoconv}: 
In view of \eqref{EMBED5}, we estimate 
\begin{equation} \label{EPhi}
\mathbb{E}\big[\|\mathcal{N}_1(\eta \Phi, u)\|_{X^{-\al, 1- \al}}\big]
\end{equation}
via the stochastic analysis from Section 4.
\end{itemize}  

\begin{remark} \rm
In fact, we do not need to take an expectation in \eqref{EPhi}
since we establish local well-posedness pathwise in $\omega$, i.e. for almost every {\it fixed} $\omega$.
Nonetheless, we estimate \eqref{EPhi} with the expectation
since it shows how $F^N_1$ and $F^N_2$ defined in \eqref{F_N}
arise along with their estimates.
\end{remark}

\medskip

\noi
$\bullet$ {\bf Estimate on (i):} 
In \cite{BO3}, the parameter $\g = \g(\al)$,  subject to the conditions (2.43) and (2.60) in \cite{BO3},  
played a certain role in estimating $R_\al$
along with the a priori bound \eqref{BOO}.
However, it plays no role in our analysis.
By Cauchy-Schwarz and Young's inequalities, we have
\begin{align*}
\eqref{262} & \leq \sum_n \| d(n, \cdot)\|_{L^2_\tau}
\jb{n}^{-1-\al} \| \ft{u}(-n, \tau_2)\|_{L^{\frac{6}{5}}_{\tau_2}}
\| \ft{u}(n, \tau_3)\|_{L^{\frac{6}{5}}_{\tau_3}}
\| \ft{u}(n, \tau_4)\|_{L^{\frac{6}{5}}_{\tau_4}}
\intertext{By H\"older inequality (with appropriate $\pm$ signs) and the fact that $-1-\al < -3\al$,}
 & \leq \sum_n \| d(n, \cdot)\|_{L^2_\tau}
\prod_{j = 2}^4 \jb{n}^{-\al-}\|\s_j^{-\al}\|_{L^3_{\tau_j}}
\| \s_j^{\al} \ft{u}(\pm n, \tau_j)\|_{L^{2}_{\tau_j}}\\
 & \leq \| d(\cdot, \cdot)\|_{L^2_{n, \tau}}
\|u\|_{X^{-\al, \al}_{6, 2}}^3
\leq \|u\|_{X^{-\al, \al}_{p, 2}}^3,
\end{align*}

\noi
where the last two inequalities follow by choosing  $\al > \frac{1}{3}$ and $p = 2+ <6$.

\medskip

\noi
$\bullet$ {\bf Estimate on (ii):} 
We use the notation from the proof of Proposition \ref{PROP:stoint}.
It follows from \eqref{GNT2} and $\eta(t) \Phi(\cdot, t) = S(t) g(\cdot, t)$
that 
\begin{align*}
(\eta \Phi)^\wedge(n_1, \tau_1) = \ft{g}(n_1, \tau_1 - n_1^3) 
= \tfrac{1}{\sqrt{2}}  \I^{(1)}_{n_1}(\tau_1 - n_1^3) + 
\tfrac{1}{\sqrt{2}}\I^{(2)}_{n_1}(\tau_1 - n_1^3).
\end{align*}

\noi
Recall that $\s_1 = \jb{\tau_1 - n_1^3} \gtrsim \jb{n n_1 n_2}$.
Also, recall from the proof of Proposition \ref{PROP:stoint}
that $|\phi_{n_1}(r)| = \chi_{[0, T]}(r)$ is independent of $\omega$.

\smallskip
\noi
$\circ$ Contribution from $\I^{(1)}_{n_1}(\tau_1 - n_1^3)$:
From  \eqref{eq:I1} with \eqref{G12}, \eqref{GNT2}, and \eqref{Partialinteg},
we estimate \eqref{EPhi}  by 
\begin{align} \label{I1}
& \lesssim \mathbb{E} \bigg[\sum_{\substack{n, n_1\\n = n_1 + n_2}} \intt_{\tau = \tau_1 + \tau_2} d\tau d\tau_1
\frac{\jb{n}^{1-\al}d(n, \tau)}{\s_0^{\al }} 
\frac{1}{\s_1^2}\int_0^T |\phi_{n_1}(r)| d\beta_{n_1}(r)
\frac{\jb{n_2}^{1-\al}c(n_2, \tau_2)}{\s_2^\al}  \bigg]
\intertext{By Cauchy-Schwarz inequality in $\omega$ and Ito isometry,}
& \lesssim  \sum_{\substack{n, n_1\\n = n_1 + n_2}} \intt_{\tau = \tau_1 + \tau_2} d\tau d\tau_1
\frac{d(n, \tau)}{\s_0^{\al }} 
\frac{\|\phi_{n_1}\|_{L^2[0, T]}}{\s_1^{\frac{3}{2} -\dl} \jb{n_1}^{\frac{1}{2}+\dl}}
\frac{\|c(n_2, \tau_2)\|_{L^2(\Omega)}}{\s_2^\al}   \label{I11}
\intertext{By $L^4_{x, t}, L^2_{x, t}, L^4_{x, t}$-H\"older inequality along 
with Lemma \ref{LEM:L4}, \eqref{EMBED2}, \eqref{phiphi}, \eqref{CN},  and \eqref{EMBED5}}
&  \lesssim T^\theta \|d\|_{L^2_{n, \tau}} \|\phi\|_{L^2([0, T];H^{-\frac{1}{2}-\dl})}
\|c\|_{L^2(\Omega;L^2_{n, \tau})}
\leq T^\theta \|\phi\|_{L^p([0, T];\ft{b}^{-\al}_{p, \infty})} 
\|u\|_{L^2(\Omega;X^{-(1-\al), \al})} \notag \\
& \lesssim T^\theta \|\phi\|_{L^p([0, T];\ft{b}^{-\al}_{p, \infty})} 
\|u\|_{L^2(\Omega;X^{-\al, \al}_{p, 2})}. \notag
\end{align}

\begin{remark} \rm
Strictly speaking, we need to take the supremum over $\{\|d\|_{L^2_{n, \tau}} =1\}$ 
{\it inside} the expectation in \eqref{I1}.
However, we do not worry about this issue for simplicity of the presentation,
since we have 
\begin{align*}
 \eqref{EPhi} & \leq \|\mathcal{N}_1(\eta \Phi, u)\|_{L^2(\Omega;X^{-\al, 1-\al})} \\
 & \leq  \bigg( \sum_n \int \frac{\jb{n}^{2-2\al}}{\s_0^{2\al }}  
 \mathbb{E} \Big| \int_0^T |\phi_{n_1}(r)| \sum_{n = n_1 + n_2} \intt_{\tau = \tau_1 + \tau_2}
\frac{\jb{n_2}^{1-\al}c(n_2, \tau_2)}{\s_1^2\s_2^\al} d\tau_1  d\beta_{n_1}(r)  \Big|^2 d\tau \bigg)^\frac{1}{2}  \\
& =  \sup_{\|d\|_{L^2_{n, \tau}} =1} \eqref{I11}
\end{align*}

\noi
by Ito isometry. 
Also, 
recall that we have
$\I^{(1)}_{n_1}(\tau_1 - n_1^3) = \int_0^T G^{(1)}_{n_1}(r, \tau_1 -n_1^3) d \beta_{n_1}(r)$
where $G^{(1)}_n(r, \tau)$ is defined in \eqref{G12}.
Hence, strictly speaking,  we should replace $G^{(1)}_{n_1}(r, \tau_1 -n_1^3)$ by $\s_1^{-2} |\phi_{n_1}(r)|$
in \eqref{I1} only after the application of Ito isometry.
Once again, we do not worry about this issue
for simplicity of the presentation.
The same remark applies in the following as well.

\end{remark}

\smallskip
\noi
$\circ$ Contribution from $\I^{(2)}_{n_1}(\tau_1 - n_1^3)$:

First, suppose that $\max(\s_0, \s_2) \gtrsim \jb{n n_1 n_2}^{\frac{1}{100}}$.
Say $\s_0 \geq \jb{n n_1 n_2}^{\frac{1}{100}}$.
Then, 
\eqref{EPhi} is estimated by 
\begin{align} \label{II2} 
& \lesssim \mathbb{E} \bigg[ \sum_{\substack{n, n_1\\n = n_1 + n_2}} \intt_{\tau = \tau_1 + \tau_2} d\tau d\tau_1
\frac{\jb{n}^{1-\al}d(n, \tau)}{\s_0^{\al }} 
\frac{1}{\s_1} \int_0^T|\phi_{n_1}(r)| d \beta_{n_1}(r)
\frac{\jb{n_2}^{1-\al}c(n_2, \tau_2)}{\s_2^\al} \bigg] \notag\\
& \lesssim  \sum_{\substack{n, n_1\\n = n_1 + n_2}} \intt_{\tau = \tau_1 + \tau_2} d\tau d\tau_1
\frac{d(n, \tau)}{\s_0^{\al -200\dl }} 
\frac{\|\phi_{n_1}\|_{L^2[0, T]}}{\s_1^{\frac{1}{2} +\dl} \jb{n_1}^{\frac{1}{2}+\dl}}
\frac{\|c(n_2, \tau_2)\|_{L^2(\Omega)}}{\s_2^\al}  
\end{align}

\noi
Then, we can conclude this case as before
by $L^4_{x, t}, L^2_{x, t}, L^4_{x, t}$-H\"older inequality as long as $\al - 200\dl > \frac{1}{3}$, which can be guaranteed by
taking $\dl> 0$ sufficiently small, or equivalently, taking $p > 2$ sufficiently close to 2.

Hence, assume $\max(\s_0, \s_2) \ll \jb{n n_1 n_2}^{\frac{1}{100}}$.
Recall the following lemma from \cite[(7.50) and Lemma 7.4]{CKSTT4}.

\begin{lemma} \label{LEM:closetocurve}
Let 
\begin{equation} \label{OMG} 
 \Omega(n) = \{ \eta \in \R : \eta = -3 n n_1 n_2 + o(\jb{n n_1 n_2}^\frac{1}{100}) \text{ for some } n_1 \in \mathbb{Z} 
\text{ with } n = n_1 + n_2 \}. 
\end{equation}

\noindent
Then, we have
\begin{equation} \label{closetocurve}
 \int \jb{\tau - n^3}^{-\frac{3}{4}} \chi_{\Omega(n)} (\tau - n^3) d \tau \lesssim 1. 
 \end{equation}
\end{lemma}

\noi
Note that 
\eqref{closetocurve} is stated with $\jb{\tau - n^3}^{-1}$ in \cite{CKSTT4}.
However, by examining the proof of Lemma 7.4 in \cite{CKSTT4}, 
one immediately sees that \eqref{closetocurve} is valid 
with $\jb{\tau - n^3}^{-\beta}$ for any $\beta > \frac{2}{3}+\frac{1}{100}$.

\smallskip

Then, 
\eqref{EPhi} is estimated by 
\begin{align*} 
& \lesssim \mathbb{E} \bigg[ \sum_{\substack{n, n_1\\n = n_1 + n_2}} \intt_{\tau = \tau_1 + \tau_2} d\tau d\tau_1
\frac{\jb{n}^{1-\al}d(n, \tau)}{\s_0^{\al }} 
\frac{\chi_{\Omega(n_1)}(\tau_1-n_1^3)}{\s_1}\int_0^T |\phi_{n_1}(r)| d\beta_{n_1}(r) 
\frac{\jb{n_2}^{1-\al}c(n_2, \tau_2)}{\s_2^\al} \bigg] 
\end{align*}

\noi
By Cauchy-Schwarz inequality and Ito isometry,
\begin{align} \label{I2}
 \lesssim  \sum_{\substack{n, n_1\\n = n_1 + n_2}} \intt_{\tau = \tau_1 + \tau_2} d\tau d\tau_1
\frac{d(n, \tau)}{\s_0^{\al }} 
\frac{\chi_{\Omega(n_1)}(\tau_1-n_1^3) \|\phi_{n_1}\|_{L^2[0, T]}}{\s_1^{\frac{1}{2} -\dl} \jb{n_1}^{\frac{1}{2}+\dl}}
\frac{\|c(n_2, \tau_2)\|_{L^2(\Omega)}}{\s_2^\al}   
\end{align}

\noi
By $L^4_{x, t}, L^2_{x, t}, L^4_{x, t}$-H\"older inequality along 
with Lemmata \ref{LEM:L4}, \ref{LEM:closetocurve}, \eqref{EMBED2}, \eqref{phiphi}, 
\eqref{CN},  and \eqref{EMBED5}, 
\begin{align*}  
& \lesssim T^\theta \|d\|_{L^2_{n, \tau}}
\big\| \jb{n_1}^{-\frac{1}{2}-\dl} \|\phi_{n_1}\|_{L^2[0, T]}
\|\chi_{\Omega(n_1)}(\tau_1-n_1^3) 
\s_1^{-\frac{1}{2}+\dl}\|_{L^2_\tau} \big\|_{L^2_n}
\|c\|_{L^2(\Omega;L^2_{n, \tau})}\\
& \leq T^\theta \|\phi\|_{L^2([0, T];H^{-\frac{1}{2}-\dl})} \|u\|_{L^2(\Omega;X^{-(1-\al), \al})} 
\lesssim T^\theta \|\phi\|_{L^p([0, T];\ft{b}^{-\al}_{p, \infty})} 
\|u\|_{L^2(\Omega;X^{-\al, \al}_{p, 2})}. \notag 
\end{align*}

\medskip

 Now, we are ready to prove Theorem \ref{THM1}.

\begin{proof}[Proof of Theorem \ref{THM1}]
 
Fix mean zero $u_0 \in \ft{b}^{-\al'}_{p, \infty}(\T)$ and $\phi$ as in \eqref{PHI},
where $\al' = \frac{1}{2}-\dl-$ with $\frac{p-2}{4p} < \dl < \frac{p-2}{2p}$
such that $(-\al')p < -1$.
Consider sequences of initial data $u_0^N \in L^2(\T)$ and 
diagonal covariance operator $\phi^N \in HS(L^2;L^2)$, given by
\begin{equation} \label{APPROX}
u^N_0 = \mathbb{P}_{\leq N} u_0 = \sum_{|n| \leq N} \ft{u}_0(n) e^{inx}
\text{ and } \phi^N (t, \omega):= \text{diag} (\phi_n (t, \omega) ; 0< |n| \leq N)
\end{equation}

\noi
where $\phi_n$ is given in \eqref{PHI}.
Now, fix $\al = \frac{1}{2}-\dl > \al'$ as in \eqref{Isb}.
Note that such $u^N_0$ converges to $u_0$ in $\mathcal{F}L^{-\al, p}(\T)$,
and thus in $\ft{b}^{-\al}_{p, \infty}(\T)$.
Also, $\phi^N$ converges to $\phi$ in $\mathcal{F}L^{-\frac{1}{2}-, p}(\T)$
for each $t$ and $\omega$, 
and thus in $\ft{b}^{-\frac{1}{2}-}_{p, \infty}(\T)$.
Then, by Monotone Convergence Theorem, 
$\phi^N$ converges to $\phi$ 
in $L^p([0, 1]\times \Omega;\ft{b}^{-\frac{1}{2}-}_{p, \infty}) $.
(Indeed, the convergence is in $L^\infty([0, 1]\times \Omega; \ft{b}^{-\frac{1}{2}-}_{p, \infty})$, 
since we have $|\phi_n(t, \omega)| = 1$ for all $n$, independent of $t \in \R$ and $\omega \in \Omega$.)
Note that a slight loss of the regularity $-\al < -\al'$
was necessary since $u^N_0$ defined in \eqref{APPROX}
does not necessarily converge to $u_0$ in $\ft{b}^{-\al'}_{p, \infty}(\T)$
due to the $L^\infty$ nature of the norm over the dyadic blocks.
We can avoid such a loss of the regularity if we start with
$u_0 \in \mathcal{F}L^{s, p}(\T)$.

Now, let $\G^N = \G^N_{u^N_0}$ be the map defined by
\begin{equation} \label{GAMMA}
\G^N v = \G^N_{u^N_0}v := S(t) u^N_0 - \tfrac{1}{2}\mathcal{N}(v, v) + \eta \Phi^N,
\end{equation}

\noi
where $\Phi^N$ is the stochastic convolution defined in \eqref{stoconv} 
with the covariance operator $\phi^N$.
By the well-posedness result in \cite{DDT1}, 
there exists a unique global solution $u^N \in L^\infty(\R^+;L^2(\T))
\cap C(\R^+;B^{0-}_{2, 1}(\T))$ a.s.
to \eqref{GAMMA} for each $N$
since $\phi^N \in HS(L^2;L^2)$.

Now, we put all the estimates together.
Note that all the implicit constants are independent of $N$.
Also, when there is no superscript $N$, 
it means that $N = \infty$.
From Lemma \ref{LEM:linear1}, we have 
\begin{equation} \label{WT0}
\| S(t) u^N_0\|_{X^{s, b, T}_{p, 2}} \leq C_1 \|u^N_0\|_{\ft{b}^{s}_{p, \infty}}
\end{equation}

\noi
for any $s, b\in \R$ with $C_1 =  C_1(b)$.
In particular,  by taking $b > \frac{1}{2}$, we see that 
$S(t) u_0$ is continuous on $[0, T]$ with values in $ \ft{b}^{s}_{p, \infty}$.
Also, by taking $b < \frac{1}{2}$, we gain a power of $T$.
From the definition of $\mathcal{N}_j(\cdot, \cdot)$ and \eqref{Np0},
we have
\begin{align}  \label{WT1}
\|\mathcal{N}(u^N, u^N)\|_{X^{-\al, \al, T}_{p, 2}}
\leq C_2 T^{\theta_1} \|u^N\|^2_{X^{-\al, \al, T}_{p, 2}}
+ 2 \|\mathcal{N}_1(u^N, u^N)\|_{X^{-\al, \al, T}_{p, 2}}.
\end{align}

\noi
Also, from \eqref{Isb} and \eqref{EMBED5}, we have
\begin{align} \label{WT2}
\|\mathcal{N}_1(u^N, u^N)\|_{X^{-\al, 1- \al, T}_{p, 2}} \leq I^N_{-\al, 1-\al}.
\end{align}

Recall that $\eta \Phi \in X^{-\al, \al}_{p, 2}$ a.s. 
from Proposition \ref{PROP:stoint}. 
Moreover, by defining $F^N_1$ and $F^N_2$ on $\T\times \R \times \Omega$ 
via their Fourier transforms:  
\begin{align} \label{F_N}
\ft{F^N_1}(n, \tau)  & = \jb{n}^{- \frac{1}{2}-\dl} (\s_0^{-\frac{3}{2}+\dl} + \s_0^{-\frac{1}{2}-\dl}) 
\int_0^T |\phi_n(r)|d\beta_n(r),
\ \text{ and } \\
\ft{F^N_2}(n, \tau) &  = \jb{n}^{- \frac{1}{2}-\dl} \chi_{\Omega(n)} (\tau - n^3)
 \s_0^{-\frac{1}{2}+\dl} \int_0^T |\phi_n(r)| d\beta_n(r)
\notag
\end{align}

\noi
for $|n| \leq N$, 
we have $F^N_1, F^N_2 \in L^2(\Omega; L^2_{x, t})$ by Ito isometry and Lemma \ref{LEM:closetocurve},
which is basically shown in the estimate on (ii).
See \eqref{I11} and \eqref{I2}.
Then, from \eqref{N_1} and the estimates on (i) and (ii), we have 
\begin{align} \label{XYZ}
I^N_{-\al, 1-\al} \leq C_3\big( T^{\theta_2} \|u^N\|_{X^{-\al, \al, T}_{p, 2}}  I^N_{-\al, 1-\al}  
+ T^{\theta_3} \|u^N\|_{X^{-\al, \al, T}_{p, 2}}^3 + T^{\theta_4} 
L^N_\omega  \|u^N\|_{X^{-\al, \al, T}_{p, 2}}\big), 
 \end{align}

\noi
where
$L^N_\omega = L^N( F^N_1, F^N_2)(\omega) :=
 \|F^N_1(\omega)\|_{L^2_{x, t}} + \|F^N_2(\omega)\|_{L^2_{x, t}}
< \infty $
a.s.
Moreover, $L^N_\omega$ is non-decreasing in $N$.
 
For fixed $R > 0$, 
choose $T>0$ small such that $C_3T^{\theta_2} R \leq \frac{1}{2}$.
Then, from \eqref{XYZ}, we have 
\begin{align} \label{WT4}
I^N_{-\al, 1-\al} \leq 2C_3 \big(
T^{\theta_3} \|u^N\|_{X^{-\al, \al, T}_{p, 2}}^3 + T^{\theta_4} L^N_\omega \|u^N\|_{X^{-\al, \al, T}_{p, 2}}\big), 
 \end{align}

\noi
for $\|u^N\|_{X^{-\al, \al, T}_{p, 2}} \leq R$.
From \eqref{GAMMA}$\sim$\eqref{WT4}, we have
\begin{align} \label{GAMMA1}
\| u^N \|_{X^{-\al, \al, T}_{p, 2}}
&  = \| \G^N u^N \|_{X^{-\al, \al, T}_{p, 2}} 
  \leq C_1 \|u^N_0\|_{\ft{b}^{-\al}_{p, \infty}}
+ \tfrac{1}{2} C_2 T^{\theta_1} \|u^N\|^2_{X^{-\al, \al, T}_{p, 2}}\notag \\
& \hphantom{XX} + 2C_3 \big(
T^{\theta_3} \|u^N\|_{X^{-\al, \al, T}_{p, 2}}^3 + T^{\theta_4} L^N_\omega \|u^N\|_{X^{-\al, \al, T}_{p, 2}}\big)
+ C_4 \|\eta \Phi^N(\omega)\|_{X^{-\al, \al}_{p, 2}}, 
\end{align}

\noi
and
\begin{align} \label{GAMMA2}
\| u^N - & u^M\|_{X^{-\al, \al, T}_{p, 2}}  
 = \| \G^N u^N - \G^M u^M\|_{X^{-\al, \al, T}_{p, 2}}  \notag \\
 & \leq  C_1 \|u^N_0- u^M_0\|_{\ft{b}^{-\al}_{p, \infty}}
+  \tfrac{1}{2} C_2 T^{\theta_1} (\|u^N\|_{X^{-\al, \al, T}_{p, 2}} 
+\|u^M\|_{X^{-\al, \al, T}_{p, 2}})\|u^N - u^M\|_{X^{-\al, \al, T}_{p, 2}}\notag  \\
& \hphantom{XX}
+ C_5 T^{\theta_3} 
\big(\|u^N\|_{X^{-\al, \al, T}_{p, 2}}^2 + \|u^M\|_{X^{-\al, \al, T}_{p, 2}}^2\big) 
 \|u^N -u^M \|_{X^{-\al, \al, T}_{p, 2}} \\
& \hphantom{XX}
+ 2C_3T^{\theta_4} L^N_\omega \|u^N -u^M \|_{X^{-\al, \al, T}_{p, 2}}
+ 2C_3T^{\theta_4} \wt{L}^{N, M}_\omega \|u^M \|_{X^{-\al, \al, T}_{p, 2}} \notag \\
&  \hphantom{XX}+ C_4 \|\eta (\Phi^N- \Phi^M)\|_{X^{-\al, \al}_{p, 2}}, 
\notag 
\end{align}

\noi
where
\begin{equation}\label{LL1}
\wt{L}^{N, M}_\omega 
:= 
 \|F^N_1-F^M_1\|_{L^2_{x, t}} + \|F^N_2-F^M_2\|_{L^2_{x, t}}.
\end{equation}

Note that in estimating the difference $\G^N u^N - \G^M u^M$ on $A_1$,  one needs to consider 
\begin{equation}\label{GAMMA3}
 \wt{I}_{-\al, 1-\al} := \| \mathcal{N}_1(u^N, u^N) - \mathcal{N}_1(u^M, u^M)\|_{-\al, 1-\al} 
\end{equation}

\noi
as in \cite{BO3}.
We can follow the argument on pp.135-136 in \cite{BO3},
except for $R_\al$ defined in \eqref{262}, yielding the third term on the right hand side of \eqref{GAMMA2}. 
As for $R_\al$, we can write 
\begin{align} \label{WTN1}
\mathcal{N}(\mathcal{N}(u,u), u)- \mathcal{N}(\mathcal{N}(v,v), v) 
= \mathcal{N}(\mathcal{N}(u+v,u-v),u)
+ \mathcal{N}(\mathcal{N}(v,v), u - v)
\end{align}

\noi as in (3.4) in \cite{BO3},
and then we can repeat the computation done for $R_\al$ in Estimate on (i), 
also yielding the third term on the right hand side of \eqref{GAMMA2}.

By definition of $u^N_0$, we have $2C_1\|u^N_0\|_{\ft{b}^{-\al}_{p, \infty}} 
\leq 2C_1\|u_0\|_{\ft{b}^{-\al}_{p, \infty}} + \frac{1}{2}$
for $N$ sufficiently large.
Also, since $\phi^N$ converges to $\phi$ in 
$L^p([0, 1]\times \Omega;\ft{b}^{-\al+}_{p, \infty})$, 
it follows from Corollary \ref{COR:stoint} and the estimate on (ii)
-see \eqref{I11}, \eqref{II2},  and \eqref{I2}-
that $\mathbb{E}[\|\eta (\Phi^N -\Phi)\|_{X^{-\al, \al}_{p, 2}}] $
and $\mathbb{E}[\wt{L}^{N, \infty}_\omega]$ defined in \eqref{LL1} converge to 0.
Hence, $\|\eta (\Phi^N -\Phi)\|_{X^{-\al, \al}_{p, 2}} + \wt{L}^{N, \infty}_\omega \to 0$ a.s. after selecting a subsequence
(which we still denote with the index $N$.) 
Then, by Egoroff's theorem, given $\eps > 0$, there exists a set $\Omega_\eps$ with 
$\mathbb{P}(\Omega^c_\eps) < 2^{-1}\eps$
such that 
 $\|\eta (\Phi^N -\Phi)\|_{X^{-\al, \al}_{p, 2}} 
 +\wt{L}^{N, \infty}_\omega \to  0$ uniformly in $\Omega_\eps$.
In particular, 
$2C_4\|\eta \Phi^N\|_{X^{-\al, \al}_{p, 2}} \leq 
2C_4\|\eta \Phi\|_{X^{-\al, \al}_{p, 2}} +\frac{1}{2}$
for large $N$ uniformly on $\Omega_\eps$.
In the following, we will work on $\Omega_\eps$.

Now, let $R_\omega = 2(C_1 \|u_0\|_{\ft{b}^{-\al}_{p, \infty}}
+ C_4 \|\eta \Phi(\omega)\|_{X^{-\al, \al}_{p, 2}}) + 1$, 
and define the stopping time $T_\omega$ by
\begin{align} \label{WTX3}
 T_\omega = \inf \{ T> 0: \max( C_3 T^{\theta_2}R_\omega,
P_1(T, R_\omega, \omega), P_2(T, R_\omega, \omega)
\geq \tfrac{1}{2}\},
\end{align}

\noi
where
\[\begin{cases}
P_1(T, R_\omega, \omega) =
\tfrac{1}{2}C_2 T^{\theta_1} R_\omega 
+ 2 C_3 T^{\theta_3} (R_\omega )^2
+ 2 C_3 T^{\theta_4 }L_\omega, & \text{ from \eqref{GAMMA1}}\\
P_2(T, R_\omega, \omega)=
C_2 T^{\theta_1} R_\omega
+ 2 C_5  T^{\theta_3} (R_\omega)^2
+ 2C_3  T^{\theta_4}L_\omega,  & \text{ from \eqref{GAMMA2}}.
\end{cases}
\]

\noi
The first condition in the definition of $T_\omega$ guarantees \eqref{WT4},
and hence \eqref{GAMMA1} and \eqref{GAMMA2},
for $\|u^N\|_{X^{-\al, \al, T}_{p, 2}} \leq R_\omega$.
The second condition along with \eqref{GAMMA1} indeed guarantees that 
\begin{equation} \label{WTX1}
\|u^N\|_{X^{-\al, \al, T}_{p, 2}} \leq R_\omega
\end{equation}

\noi for $T \leq T_\omega$
from the following observation.
Since we have the temporal regularity $b = \al < \frac{1}{2}$, 
we have 
$\|u^N\|_{X^{-\al, \al, T}_{p, 2}}
= \|\chi_{[0, T]}u^N\|_{X^{-\al, \al}_{p, 2}}$,
where $\chi_{[0, T]}$ denotes the characteristic function of the time interval $[0, T]$.
See Bourgain \cite{BO4}.
Hence, $\|u^N\|_{X^{-\al, \al, T}_{p, 2}}$ is continuous in $T$
since 
\begin{equation} \label{WTX2}
\big|\|u^N\|_{X^{-\al, \al, T+\dl}_{p, 2}} -\|u^N\|_{X^{-\al, \al, T}_{p, 2}}\big|
\leq \|u^N\|_{X^{-\al, \al}_{p, 2}[T, T+\dl]}
\lesssim \dl^\theta \|u^N\|_{X^{0-, \frac{1}{2}}[T, T+\dl]}
\end{equation}

\noi
for sufficiently small $\dl>0$.
Note that the last term in \eqref{WTX2} is finite for small $\dl$
since the local-in-time solutions constructed in \cite{DDT1} 
are controlled in this norm (indeed in a stronger norm
adapted to the Besov space $B_{2, 1}^{0-}$.) 
Then, \eqref{WTX1} follows from \eqref{GAMMA1}, 
the second condition in \eqref{WTX3}, and the continuity of the norm in $T$
since \eqref{WTX1} clearly holds at $T = 0$.

From \eqref{GAMMA2} along with the third condition in \eqref{WTX3}, 
we have
\begin{align} \label{WTX4}
\| u^N -  u^M\|_{X^{-\al, \al, T_\omega}_{p, 2}}  
   \leq  \ & 2C_1 \|u^N_0- u^M_0\|_{\ft{b}^{-\al}_{p, \infty}}
+ 4C_3T^{\theta_4} R_\omega \wt{L}^{N, M}_\omega  \\
 & + 2C_4 \|\eta (\Phi^N- \Phi^M)\|_{X^{-\al, \al}_{p, 2}}. \notag 
\end{align}

\noi
The right hand side of \eqref{WTX4} goes to 0 as $N, M \to \infty$
since $u^N_0$ is Cauchy in $\ft{b}^{-\al}_{p, \infty}$
and 
$\|\eta (\Phi^N- \Phi^M)\|_{X^{-\al, \al}_{p, 2}}+ \wt{L}^{N, M}_\omega \to 0$  on $\Omega_\eps$ uniformly in $N, M$.
Let $u$ denote the limit in $X^{-\al, \al, T_\omega}_{p, 2}$.

In the following, we give a brief discussion  to show that the limit $u$ is a solution to \eqref{duhamel1}.
Clearly, $S(t) u_0^N$ and $\eta \Phi^N$ converge
to $S(t) u_0$ and $\eta \Phi$ in $X^{-\al, \al, T_\omega}_{p, 2}$.
It follows from \eqref{Np0} that $\mathcal{N}_0(u^N, u^N)$ converges
$\mathcal{N}_0(u, u)$ in $X^{-\al, \al, T_\omega}_{p, 2}$.
In view of \eqref{WT4}, \eqref{GAMMA2}, and \eqref{GAMMA3}, 
we see that $\mathcal{N}_j(u^N, u^N)$ is Cauchy in 
a slightly stronger space $X^{-\al, 1- \al, T_\omega}_{p, 2}$, $j = 1, 2$. 
Let $v_j$ denote the corresponding limit.
Thus, from \eqref{GAMMA}, we have
\begin{equation} \label{UU1}
u = S(t)u_0 -\tfrac{1}{2} \mathcal{N}_0(u, u)
-\tfrac{1}{2} (v_1 +v_2) + \eta \Phi.
\end{equation}

\noi
Now, we need
 to show that $\mathcal{N}_j(u^N, u^N)$ indeed converges to $\mathcal{N}_j(u, u)$,
$j = 1, 2$.
By symmetry, we only consider 
$\mathcal{N}_1(u, u) - \mathcal{N}_1(u^N, u^N)$.
As before, we substitute \eqref{UU1} (and \eqref{GAMMA})
in the first factor $u$ (and $u^N$) of $\mathcal{N}_1(\cdot, \cdot)$,
respectively.
There are three contributions to consider.

\smallskip
\noi
$\bullet$ {\bf (A)} Contribution from the stochastic terms: 
We have
\begin{align} \label{UU2}
\mathcal{N}_1(\eta\Phi, u) - \mathcal{N}_1(\eta\Phi^N, u^N)
=\mathcal{N}_1(\eta(\Phi - \Phi^N), u)
+\mathcal{N}_1(\eta\Phi^N, u -u^N).
\end{align}

\noi
From Estimate on (ii), we have
\[\|\eqref{UU2}\|_{ X^{-\al, \al, T_\omega}_{p, 2}}
\lesssim 
 \wt{L}^{N, \infty}_\omega \|u \|_{X^{-\al, \al, T_\omega}_{p, 2}}
+  L^N_\omega \|u^N -u \|_{X^{-\al, \al, T_\omega}_{p, 2}}
\to 0
\]

\noi
as $N\to \infty$, since $\|u \|_{X^{-\al, \al, T}_{p, 2}} \leq R_\omega$
and $\wt{L}^{N, \infty}_\omega \to 0$ uniformly on $\Omega_\eps$.

\smallskip
\noi
$\bullet$ {\bf (B)} Contribution from $\mathcal{N}_0(\cdot, \cdot)$:
In this case, we consider
\begin{equation} \label{UU3}
\mathcal{N}_1(\mathcal{N}_0(u, u), u) -\mathcal{N}_1(\mathcal{N}_0(u^N, u^N), u^N).
\end{equation}

\noi
Note that we have $\s_1 \geq \s_0, \s_2, \s_3, \s_4$
from the definition of $\mathcal{N}_1(\cdot, \cdot)$ and $\mathcal{N}_0(\cdot, \cdot)$.
See \eqref{eq:u_1} and \eqref{eq:I2}.
Indeed, we have $\s_1 \geq \s_0, \s_2$
since we are on $A_1$ defined in \eqref{AJJ},
and also $\s_1 \geq \s_3, \s_4$ since we are on the support of $\mathcal{N}_0(\cdot, \cdot)$
in the first factor of $\mathcal{N}_1(\cdot, \cdot)$.
Once again, one can easily follow the argument on p.136 in \cite{BO3}
and show
\begin{equation*}
\|\eqref{UU3}\|_{X^{-\al, \al, T_\omega}_{p, 2}}
\lesssim 
\big(\|u^N\|_{X^{-\al, \al, T_\omega}_{p, 2}}^2 + \|u\|_{X^{-\al, \al, T_\omega}_{p, 2}}^2\big) 
 \|u^N -u \|_{X^{-\al, \al, T_\omega}_{p, 2}} \to 0.
\end{equation*}

\noi
In treating $R_\al - R_\al^N$
defined in \eqref{262}, 
one needs to proceed as before, using \eqref{WTN1}
and Estimate on (i).

\smallskip
\noi
$\bullet$ {\bf (C)} Contribution from $v_j$ and $\mathcal{N}_j(u^N, u^N)$,
$j = 1$ or $2$:
By symmetry, assume $j = 1$.
In this case, we have
$\s_1 \geq \s_0, \s_2$ but $\s_3 \geq \s_1, \s_4$.
i.e. we control \eqref{N_1} by 
the first term on the right hand side.
See (II.1) on p.126 in \cite{BO3}. 
Now, we need to estimate
\begin{align} \label{UU4}
\mathcal{N}_1(v_1, u) - & \mathcal{N}_1(\mathcal{N}_1(u^N, u^N), u^N) \notag \\
& = \mathcal{N}_1(  v_1- \mathcal{N}_1(u^N, u^N), u) 
 +\mathcal{N}_1(\mathcal{N}_1(u^N, u^N), u - u^N)
=: \I + \II
.
\end{align}

\noi
Then, by proceeding as in \cite{BO3}
with  \eqref{EMBED5} and  \eqref{WT4}, we have
\begin{align*}
\|\, \II \, \|_{X^{-\al, 1- \al, T_\omega}_{p, 2}}
\lesssim 
I^N_{-\al, 1-\al}
\|u - u^N\|_{X^{-(1-\al), \al, T_\omega}}
\lesssim \|u - u^N\|_{X^{-\al, \al, T_\omega}_{p, 2}} \to 0.
\end{align*}

\noi
By proceeding as in (II.1)  in \cite{BO3}
with $|n_1|^\al$ replaced by $|n_1|^{1-\al}$,
followed by \eqref{EMBED5}, we have
\begin{align*}
\|\, \I \, \|_{X^{-\al, 1- \al, T_\omega}_{p, 2}}
& \lesssim 
\|v_1 - \mathcal{N}_1(u^N, u^N)\|_{-(1-\al), 1-\al}
\|u \|_{-(1-\al), \al} \\
& \lesssim \|v_1 - \mathcal{N}_1(u^N, u^N)\|_{X^{-\al, 1- \al, T_\omega}_{p, 2}}
\|u \|_{X^{-\al, \al, T_\omega}_{p, 2}} \to 0
\end{align*}

\noi
since $v_1 = \lim_{N\to \infty} \mathcal{N}_1(u^N, u^N)$ in $X^{-\al, 1- \al, T_\omega}_{p, 2}$
by definition.

\smallskip
Hence,  we have  $u = \G_{u_0} u$
for each $\omega \in \Omega_\eps$.
i.e. $u$ is a mild solution to \eqref{SKDV1} on $[0, T_\omega]$. 
Let $\Omega^{(1)} = \Omega_\eps$.
Now, we can recursively construct $\Omega^{(j+1)} \subset \Omega \setminus \bigcup_{k = 1}^j\Omega^{(k)}$
for $j = 1, 2, \cdots$
with 
$\mathbb{P}(\Omega \setminus \bigcup_{k = 1}^j\Omega^{(k)}) < 2^{-j} \eps$
such that 
$\|\eta (\Phi^N -\Phi)\|_{X^{-\al, \al}_{p, 2}} $
and 
 $\wt{L}^{N, \infty}_\omega$ converge to  $0$ uniformly in each $\Omega^{(j)}$.
Then, by repeating the argument, we can construct a solution $u$ on 
$\bigcup_{j = 1}^\infty\Omega^{(j)}$.
Note that 
$\mathbb{P}(\Omega \setminus \bigcup_{j = 1}^\infty\Omega^{(j)}) = 0$.

\smallskip

We have constructed a solution $u$ to \eqref{SKDV1} in $X^{-\al, \al, T_\omega}_{p, 2}$
with $u_0 \in \ft{b}^{-\al'}_{p, \infty}$.
Since $u$ is a solution, the a priori estimate \eqref{GAMMA1} holds
with the regularity $(s, b) = (-\al', \al')$ in place of $(-\al, \al)$.
Then, we easily see that $u \in X^{-\al', \al', T_\omega}_{p, 2}$,
by redefining $R_\omega$ and $T_\omega$ with this regularity.  
In the remaining of the paper, 
we  work only with the spatial regularity $s = -\al'$,
i.e. there is no approximating sequences any more.
Hence, for notational simplicity, we will use $-\al$ in place of $-\al'$
to denote the spatial regularity of the solution in the following.

We still need to take care of  several issues.
Note that the temporal regularity $b = \al = \frac{1}{2}-\dl$ of the solution $u$ is less than $\frac{1}{2}$.
In particular, we need to show that the solution $u$ is continuous 
from $[0, T_\omega]$ into $\ft{b}^{-\al}_{p, \infty}$.
We also need to show its uniqueness and continuous dependence on the initial data.

From Proposition \ref{PROP:stoconti}, 
$\eta \Phi \in C([0, T_\omega]; \ft{b}^{-\al}_{p, \infty}$) a.s.
Also, it follows from \eqref{WT0} with $b = \frac{1}{2}+\dl$, 
\eqref{WT2}, \eqref{WT4}, and symmetry on $\s_1$ and $\s_2$, 
that 
\[S(t)u_0 + \mathcal{N}_1(u, u) + \mathcal{N}_2(u, u)
\in X^{-\al, \frac{1}{2}+\dl, T_\omega}_{p, 2} \subset C([0, T_\omega]; \ft{b}^{-\al}_{p, \infty})\] 

\noi
a.s.
Now, we consider 
$\mathcal{N}_0(u, u)$, i.e. when $\s_0 = \MAX$.
Note that the contribution comes only from $\mathcal{M}_2(u, u)$ defined in \eqref{Duhamel}.
Let $\mathcal{N}_3(u, u)$
denotes the contribution of $\mathcal{N}_0(u, u)$
on $\{\max(\s_1, \s_2) \gtrsim \jb{n n_1 n_2}^\frac{1}{100}\}$,
and $\mathcal{N}_4(u, u) = \mathcal{N}_0(u, u) - \mathcal{N}_3(u, u)$.

\smallskip

\noi
$\bullet$ {\bf Case (a):} First, we consider $\mathcal{N}_3(u, u)$.
i.e. $\max(\s_1, \s_2) \gtrsim \jb{n n_1 n_2}^\frac{1}{100}$.
Say  $\s_1 \gtrsim \jb{n n_1 n_2}^\frac{1}{100}$.
Then, by Lemma \ref{LEM:linear2} and  \eqref{EMBED1},  we have
\begin{align*} 
\| \mathcal{N}_3(u, & u) \|_{X^{-\al, \frac{1}{2}+\dl, T_\omega}_{p, 2}}
 \lesssim \| \dx (u^2) \|_{X^{-\al, -\frac{1}{2}+\dl, T_\omega}_{p, 2}} 
\lesssim  \| \dx (u^2) \|_{X^{-\al, -\frac{1}{2} + \dl,T_\omega}}
\end{align*}

\noi
Then, by duality and \eqref{MAXMAX}, we have
\begin{align*}
& = \sup_{\|d\|_{L^2_{n, \tau}} = 1} 
 \sum_{\substack{n, n_1\\n = n_1 + n_2}} \intt_{\tau = \tau_1 + \tau_2} 
\frac{\jb{n}^{1-\al}d(n, \tau)}{\s_0^{\frac{1}{2}-\dl} } 
\prod_{j = 1}^2 \frac{\jb{n_j}^{1-\al}c(n_j, \tau_j)}{\s_j^\al}  d\tau d\tau_1\\
& \lesssim  \sup_{\|d\|_{L^2_{n, \tau}} = 1} 
\sum_{\substack{n, n_1\\n = n_1 + n_2}} \intt_{\tau = \tau_1 + \tau_2} 
d(n, \tau)
\frac{c(n_1, \tau_1)}{\s_1^{\al- 200\dl}}
\frac{c(n_2, \tau_2)}{\s_2^\al} d\tau d\tau_1
\end{align*}

\noi
where $c(n, \tau)$ is defined in \eqref{CN}.
Then, by $L^2_{x, t}, L^4_{x, t}, L^4_{x, t}$-H\"older inequality along 
with Lemma \ref{LEM:L4}, \eqref{CN},  and \eqref{EMBED5},
\begin{align*}
& \leq \|c\|^2_{L^2_{n, \tau}}
\leq  \|u\|^2_{X^{-(1-\al), \al}} 
\lesssim \|u\|^2_{X^{-\al, \al}_{p, 2}} < \infty.
\end{align*}

\noi
$\bullet$ {\bf Case (b):} Now, consider $\mathcal{N}_4(u, u)$.
i.e. $\max(\s_1, \s_2) \ll \jb{n n_1 n_2}^\frac{1}{100}$.
Note that it suffices to show that 
$\mathcal{N}_0(u, u) \in X^{-\al, 0, T_\omega}_{p, 1}$, 
since $X^{-\al, 0, T_\omega}_{p, 1}
\subset C([0, T_\omega]; \ft{b}^{-\al}_{p, \infty})$. 
Then, by Cauchy-Schwarz inequality, Lemma \ref{LEM:closetocurve} and duality, we have 
\begin{align*} 
\| \mathcal{N}_4(u, & u) \|_{X^{-\al, 0, T_\omega}_{p, 1}}
 \leq \| \dx (u^2) \|_{X^{-\al, -1, T_\omega}_{2, 1}} 
 \leq  
\big\| \|\jb{n}^{-\al} \jb{\tau - n^3}^{-1} \chi_{\Omega(n)}(\tau - n^3)
\ft{\dx (u^2)}(n, \tau)  \|_{L^1_{\tau}} \big\|_{L^2_n}  \\
& \leq \| \jb{\tau - n^3}^{-\frac{1}{2}+\dl} \chi_{\Omega(n)}(\tau - n^3) \|_{L^2_\tau}
\|\dx(u^2)\|_{-\al, -\frac{1}{2}-\dl}\\
& \lesssim  \sup_{\|d\|_{L^2_{n, \tau}} = 1} 
 \sum_{\substack{n, n_1\\n = n_1 + n_2}} \intt_{\tau = \tau_1 + \tau_2} 
\frac{\jb{n}^{1-\al}d(n, \tau)}{\s_0^{\frac{1}{2}+\dl} } 
\prod_{j = 1}^2 \frac{\jb{n_j}^{1-\al}c(n_j, \tau_j)}{\s_j^\al}  d\tau d\tau_1\\
& \lesssim  \sup_{\|d\|_{L^2_{n, \tau}} = 1} 
\sum_{\substack{n, n_1\\n = n_1 + n_2}} \intt_{\tau = \tau_1 + \tau_2} 
d(n, \tau)
\frac{c(n_1, \tau_1)}{\s_1^\al}
\frac{c(n_2, \tau_2)}{\s_2^\al} d\tau d\tau_1.
\end{align*}

\noi
The rest follows as before.
Hence, the solution $u$ is continuous from $[0, T_\omega]$
to $\ft{b}^{-\al}_{p, \infty}$.

\medskip

Lastly, we show the uniqueness and the continuous dependence of the solutions on the initial data.
Let $u$ and $v$ be the mild solutions of \eqref{SKDV1}
on $[0, T_\omega]$
with initial data $u_0$ and $v_0$ respectively.
i.e. 
\begin{equation}\label{GAMMAGAMMA}
u - v= \G_{u_0}u -\G_{v_0} v 
= S(t) (u_0-v_0) - \tfrac{1}{2}\big(\mathcal{N}(u, u)-\mathcal{N}(v, v)\big),
\end{equation}

\noi
where $\G$ is defined in \eqref{GAMMA}.
Moreover, assume that 
\begin{align} \label{Hirano1}
\|u_0\|_{\ft{b}^{-\al}_{p, \infty}}, 
\|v_0\|_{\ft{b}^{-\al}_{p, \infty}}, 
\|u\|_{X^{-\al, \al, T_\omega}_{p, 2}},
\|v\|_{X^{-\al, \al, T_\omega}_{p, 2}} \leq R.
\end{align}

\noi
Let $\wt{\mathcal{N}}_j (u, v) := -\tfrac{1}{2}\big(\mathcal{N}_j(u, u) -\mathcal{N}_j(v, v)\big)$
for $j = 1, \cdots, 4$.
First, note that 
$\| \wt{\mathcal{N}}_4(u,  v) \|_{X^{-\al, \eps, T_\omega}_{p, 1}}
\lesssim R^2 < \infty$ from (a slight variation of) Case (b), and
we have
\begin{align*} 
\|(u - v) - \wt{\mathcal{N}}_4(u,  v) \|_{X^{-\al, \eps, T_\omega}_{p, 1}}
\leq \Big\|S(t)(u_0 - v_0) + \sum_{j = 1}^3\wt{\mathcal{N}}_j(u, v)
\Big\|_{X^{-\al, \frac{1}{2}+\dl, T_\omega}_{p, 2}}
\lesssim C_1(R) <\infty
\end{align*}

\noi
by Cauchy-Schwarz inequality with $\eps < \dl$,  
followed by \eqref{WT0}, \eqref{WT2}, \eqref{WT4}, 
 Case (a), and \eqref{Hirano1}.
Then, by interpolation and Cauchy-Schwarz inequality, we have
\begin{align} \label{Hirano2}
\|u-v  \|_{C([0, T_\omega]; \ft{b}^{-\al}_{p, \infty})}
& \lesssim \|u-v  \|_{X^{-\al, 0, T_\omega}_{p, 1}} \notag 
\lesssim \|u-v  \|^\beta_{X^{-\al, -\dl-, T_\omega}_{p, 1}}
\|u-v  \|^{1-\beta}_{X^{-\al, \eps, T_\omega}_{p, 1}}\\
& \lesssim C_2(R) \|u-v  \|^\beta_{X^{-\al, \frac{1}{2} -\dl, T_\omega}_{p, 2}}
\end{align}

\noi
with $\beta = \frac{\eps}{\eps+\dl+} \in (0, 1)$.
From \eqref{WT0} and the nonlinear estimates (see  \eqref{WT1}, \eqref{WT4}, \eqref{GAMMA2},
\eqref{GAMMA3}), we have  
\begin{align*} 
\|u-v  \|_{X^{-\al, \frac{1}{2} -\dl, T_\omega}_{p, 2}}
\lesssim 
\|u_0-v_0\|_{\ft{b}^{-\al}_{p, \infty}}
+C_3(R) T^\theta_\omega\|u-v  \|_{X^{-\al, \frac{1}{2} -\dl, T_\omega}_{p, 2}}.
\end{align*}

\noi
Hence, for sufficiently small $T>0$, we have
\begin{align}\label{Hirano3}
\|u-v  \|_{X^{-\al, \frac{1}{2} -\dl, T_\omega}_{p, 2}}
\lesssim 
 \|u_0-v_0\|_{\ft{b}^{-\al}_{p, \infty}}.
\end{align}

\noi
Therefore, 
it follows from \eqref{Hirano2} and \eqref{Hirano3}
that the solution map is H\"older continuous with the bound
\[\|u-v  \|_{C([0, T_\omega]; \ft{b}^{-\al}_{p, \infty})}
\leq C_4(R)\|u_0-v_0\|^\beta_{\ft{b}^{-\al}_{p, \infty}}.\]

\noi
In particular, the solution is unique.  This completes the proof of Theorem \ref{THM1}.
\end{proof}

\end{document}